\documentclass[12pt]{amsart}
\usepackage{amsmath,amssymb,amsfonts,amsopn,mathrsfs,bm}
\usepackage{mathtools}
\usepackage{graphicx}
\usepackage{epstopdf}
\DeclareGraphicsExtensions{.png}
\usepackage{caption}
\usepackage{subcaption}

\newcommand{\tfa}{time-frequency analysis}

\newtheorem{theorem}{Theorem}[section]

\newtheorem{proposition}[theorem]{Proposition}

\newtheorem{remark}[theorem]{Remark}
\newtheorem{Alg}[theorem]{Algorithm}
\theoremstyle{definition}

\newcommand{\beqa}{\begin{eqnarray*}}
\newcommand{\eeqa}{\end{eqnarray*}}

\newcommand{\field}[1]{\mathbb{#1}}
\newcommand{\bR}{\field{R}}        
\newcommand{\bN}{\field{N}}        
\newcommand{\bZ}{\field{Z}}        
\def\al{\alpha}                    

\def\G{\mathcal{G}}

\def\la{\lambda}

\def\th{\theta}

\def\Ghframe{\mathcal{G}^\hbar(\f,\Lambda)}

\def\cS{\mathcal{S}}

\def\cC{\mathcal{C}}

\def\a{\aleph}

\def\rd{\bR^d}

\def\rdd{{\bR^{2d}}}

\def\lrd{L^2(\rd)}

\def\zd{\bZ^d}

\def\intrd{\int_{\rd}}
\def\intrdd{\int_{\rdd}}

\def\L{\Big(}
\def\R{\Big))}

\def\<{\left<}
\def\>{\Big)>}

\def\mv1{M_v^1}

\def\phas{(x,\xi)}
\def\mn{(m,n)}
\def\mn'{(m',n')}

\def\a{\alpha}

\def\t{\tau}

\def\R{\mathbb{R}}
\def\Ren{\mathbb{R}^d}

\def\tauhz0{\widehat{\mathcal{T}}^\hbar(z_0)}
\def\tauhz{\widehat{\mathcal{T}}^\hbar(z)}

\def\f{\varphi}

\def\Sn2{S_{2}(L^{2}(\Ren))}
\def\S1{S_{1}(L^{2}(\Ren))}
\def\sig00{\sigma_{0,0}}
\def\th{\widehat{\mathcal{T}}^\hbar}
\def\t{\widehat{\mathcal{T}}}
\def\mhm{\widehat{M}_{\hbar^{-1/2}}}
\def\mhp{\widehat{M}_{\hbar^{1/2}}}
\def\st{\widehat{S_t}}
\def\sht{\widehat{S_t}^\hbar}

\def\la{\langle}
\def\ra{\rangle}

\newcommand{\A}{\mathcal{A}}

\newcommand{\Ha}{\widehat{H(t)}} 
\newcommand{\Tr}{\mathrm{Tr}} 
\begin{document}
\begin{abstract}
\medskip The present paper is devoted to the semiclassical analysis of  linear Schr\"odinger equations from a Gabor frame perspective. We consider (time-dependent) smooth Hamiltonians  with at most quadratic growth. Then we construct higher order parametrices for the corresponding Schr\"odinger equations by means of $\hbar$-Gabor frames, as recently defined by M. de Gosson, and we provide precise $L^2$-estimates of their accuracy,  in terms of the Planck constant $\hbar$. Nonlinear parametrices, in the spirit of the nonlinear approximation, are also presented.
Numerical experiments are exhibited to compare our results with the early literature.
 \medskip
\end{abstract}

\title{Gabor Frames of Gaussian Beams for the Schr\"odinger equation}

\author{Michele Berra, Martina Bulai, Elena Cordero,  and Fabio Nicola}
\address{Universit\`a di Torino, Dipartimento di Matematica, via Carlo Alberto 10, 10123 Torino, Italy}
\email{michele.berra@unito.it}
\address{Universit\`a di Torino, Dipartimento di Matematica, via Carlo Alberto 10, 10123 Torino, Italy}
\email{elena.cordero@unito.it}
\address{Universit\`a di Torino, Dipartimento di Matematica, via Carlo Alberto 10, 10123 Torino, Italy}
\email{iuliamartina.bulai@unito.it}
\address{Dipartimento di Scienze Matematiche,
Politecnico di Torino, corso Duca degli Abruzzi 24, 10129 Torino,
Italy}
\email{fabio.nicola@polito.it}

\maketitle

\section{Introduction}

The goal of this paper is  to construct asymptotic solutions for Schr\"odinger   equations
  \begin{equation}\label{C1}
\begin{cases} i \hbar
\partial_t u =\Ha u\\
u(0)=u_0,
\end{cases}
\end{equation}
by means  of Gabor frames in the semi-classical regime ($\hbar\to 0^+)$. Here $t\in[0,T]$, the initial condition $u_0\in\lrd$ and the quantum Hamiltonian $\Ha$ is supposed to be the $\hbar$-Weyl quantization of the classical observable $H(t,X)$, with $X=\phas\in\rdd$.
\subsection{Literature overview}
There are many results about asymptotic solutions for partial differential equations (PDE's), especially when the initial value is a wave packet, i.e. it is well localized in the physical space and it oscillates with an approximately constant frequency. In particular, if the initial profile is Gaussian (a coherent state), the solution will be highly concentrated along the classical trajectory, according to the correspondence principle. Such a semi-classical analysis for Schr\"odinger-type equations were widely studied in several papers, see e.g. \cite{ComberscureRobert2012,deGossonbook,deGossonACHA,hagedorn1980semiclassical,hagedorn1981semiclassical,klaudercoherent,robert2004propagation} and the textbooks \cite{combescure2012quadratic,GS2,GS1,GS3,Zworski}.

\par
The natural idea of this work is to decompose the initial value $u_0$ in \eqref{C1} by means of a $\hbar$-Gabor frame whose atoms are Gaussian coherent states, construct asymptotic solutions for each of them, a so-called \emph{Gaussian beam}, and finally by superposition obtain the asymptotic solution to \eqref{C1}. The main issues are the following:
\begin{itemize}
               \item  Construction of a parametrix via Gabor frames
                \item Estimates in $L^2$ for the parametrix and the error term
                \item Numerical results.
              \end{itemize}
Despite the simplicity of the idea, we do not know a fully rigorous treatment of this matter. There are various attempts (see e.g.\ \cite{MR3161186} and references therein) where however several arguments are carried out only at a heuristic level and with numerical experiments. The present paper is devoted to a rigorous study of these issues for a class of smooth Hamiltonians with at most quadratic growth and, unlike the previous work, we address from the beginning a finer analysis, that is higher order approximations: the approximate solution is searched as a (finite) sum of powers of $\hbar$, and the order of  approximation can be arbitrarily large.
\subsection{Notation and ($\hbar$-)Gabor frames} To be explicit, let us fix some notation.\par
  The $\hbar$-Weyl quantization of a function $H$ on the phase space $\rdd$ is formally defined by
\begin{equation}\label{hWeyl}
\widehat{H}u(x)=Op^w_{\hbar}[H]u(x)=(2\pi \hbar)^{-d}\intrdd e^{i\hbar^{-1}(x-y)p } H\Big(\frac{x+y}2,p\Big) u(y)\,dydp
\end{equation}
for every $u$ in the Schwartz space $\cS(\rd)$. The function $H$ is called the $\hbar$-Weyl symbol of $\widehat{H}$.
For $z_0=(x_0,p_0)\in\rdd$, we define by $\widehat{\mathcal{T}}^\hbar(z_0)$ the Weyl operator
\begin{equation}\label{hT}
\widehat{\mathcal{T}}^\hbar(z_0)u(x)=Op^w_{\hbar}[e^{i\hbar^{-1}(p_0 x-x_0p)}]u(x)= e^{i\hbar^{-1}(p_0x-x_0p_0/2)}u(x-x_0).
\end{equation}
Such operator meets the definition of the so-called $\hbar$-Gabor frames, introduced in \cite{deGossonACHA} as generalizations of Gabor frames.
For a given  lattice $\Lambda$ in $\rdd$ and a non-zero square integrable function $\f$ (called window) on $\rd$   the system
$$\mathcal{G}^\hbar(\f,\Lambda)=\{\tauhz\f\,:\,z\in\Lambda\}
$$
is called a $\hbar$-Gabor frame if it is a frame for $\lrd$, that is there exist constants $0<a\le  b$ such that
\begin{equation}\label{framedef}
a\|f\|_2^2\leq\sum_{z\in \Lambda}|\la f,\tauhz\f\ra|^2\leq b\|f\|_2^2,\quad \forall f \in\lrd.
\end{equation}
In particular, when $\hbar=(2\pi)^{-1}$,  the operator  $\widehat{\mathcal{T}}(z_0):=\widehat{\mathcal{T}}^{(2\pi)^{-1}}(z_0)$ is the so-called time-frequency (or phase-space) shift
$$\widehat{\mathcal{T}}(z_0)f(x)=e^{-\pi i p_0x_0} e^{2\pi i p_0 x} f(x-x_0)=T_{x_0}M_{p_0}f(x), \quad  z_0=(x_0,p_0),$$
where translation and modulation operators are defined by
 $$T_{x_0}f(x) = f(x-x_0) \quad \mbox{and}\quad M_{p_0}f(x) = e^{2\pi i p_0  x} f(x).$$
  In this case we recapture the standard definition of a  Gabor frame (see the next section for more details). Since they first appearance in the fundamental paper by Duffin and Schaffer \cite{DuffinSchaeffer79} on non-uniform sampling of band-limited functions, frames have been applied in many fields of mathematics and physics. In particular, Gabor frames have been widely used in signal analysis,  time-frequency analysis, quantum physics. Recently Gabor frames have been successfully applied in the study of PDE's. In \cite{fio3,cnr2012time} they have shown to provide optimally sparse representations for Schr\"{o}dinger type propagators and in \cite{GaborEvolEq}  reveal to be an equally efficient tool for representing solutions to hyperbolic and parabolic-type differential equations with constant coefficients.  More generally, wave packet analysis and almost diagonalization of pseudodifferential and Fourier integral operators by Gabor frames have been  performed in \cite{cnr2007time,cnr2013wave,CNRExponentiallysparse2013,cnrShubin,cnr2013schr,grochenig2008banach,nic}.\par

Pursuing the work on deformation of Gabor frames, that have been investigated  by many authors \cite{ascensi2014dilation,cnr2012approximation,feichtinger2004varying,grochenig2013deformation,luo2000deforming},  de Gosson in \cite{deGossonACHA} uses the Schr\"odinger evolution to deform Gabor frames.

This paper is intended to be in some sense complementary to the work of de Gosson \cite{deGossonACHA},  because we use the frames to construct an approximate propagator for the Schr\"odinger equation. We also adopt almost  the same notation as in \cite{deGossonACHA}.\par
\subsection{Main Results}
Here we consider $\hbar$-Gabor frames where the window function is the standard Gaussian
 \begin{equation}\label{Gaussian}
 \phi_0(x)=\pi^{-d/4}e^{-|x|^2/2}
 \end{equation}
 and its rescaled version
 \begin{equation}\label{Gaussianh}
 \phi^{\hbar}_0(x)=(\pi \hbar)^{-d/4}e^{-|x|^2/(2\hbar)}.
 \end{equation}
 We define the coherent state centered at $z\in\rdd$ the function
\begin{equation}\label{Gaussianhz}
 \phi^{\hbar}_z(x)=\tauhz \phi^{\hbar}_0(x).
 \end{equation}
 Consider the solution $z_t=(x_t,p_t)$ to the Hamiltonian system
\begin{equation}\label{HS} \dot{x}_t = \partial_p H(t,x_t, p_t),\quad \dot{p}_t = -\partial_x H(t,x_t, p_t)
\end{equation}
 with initial value $z_0=(x_0,p_0)$ and define
 \begin{equation}\label{Hquad}
 H^{(2)}_{z_0}(t,X)=\sum_{|\gamma|=2}\frac1 {\gamma!}\partial^\gamma_X H(t,z_t)X^\gamma,\quad X\in\rdd.
 \end{equation}
 It is well-known that the solution to the corresponding operator Schr\"odinger equation  for $\hbar=1$
 \begin{equation}\label{metap}
 i \partial_t \widehat{S_t}(z_0) =Op^w_1 [H^{(2)}_{z_0}(t)] \widehat{S_t}(z_0)\quad\widehat{S_0}(z_0) =I,
 \end{equation}
 is a metaplectic operator $\widehat{S_t}(z_0)$ corresponding to the symplectic matrix $S_t(z_0)$ via the metaplectic representation \cite{deGossonbook,robert2004propagation}.
Following the works \cite{hagedorn1980semiclassical,hagedorn1981semiclassical,robert2004propagation}, a natural ansatz for asymptotic solutions  to \eqref{C1}, modulo $O(\hbar^{(N+1)/2})$, $N\in\R$, where the initial value is the coherent state $\phi^{\hbar}_{z_0}$, that is
 \begin{equation}\label{C1z0}
\begin{cases} i \hbar
\partial_t u =\Ha u\\
u(0)=\phi^{\hbar}_{z_0},
\end{cases}
\end{equation}
is provided by the Gaussian beam
\begin{equation}\label{solapprox}
 \phi^{\hbar,N}_{z_0}(x)=e^{\frac {i} {\hbar}\delta(t,z_0)}\widehat{\mathcal{T}}^{\hbar}(z_t)\widehat{M}_{\hbar^{-1/2}} \widehat{S_t}(z_0)\sum_{j=0}^N \hbar^{j/2}b_j(t,x)\phi_0(x).
\end{equation}
Here the symmetrized action $\delta$ is defined by
\begin{equation}\label{symmaction}
\delta(t,z_0)=\int_0^t \Big(\frac12 \sigma(z_s,\dot{z}_s)-H(s,z_s)\Big)\,ds,
\end{equation}
with $\sigma$ being the standard symplectic form;  the metaplectic operator $\widehat{M}_{\hbar^{-1/2}}$ is given by $$\widehat{M}_{\hbar^{-1/2}}f(x)=\hbar^{-d/4}f(\hbar^{-1/2} x)$$ and the functions $b_j(t,x)$ are suitable polynomials in $x$ with coefficients depending on $t,z_0$ ($b_0\equiv1$), as we shall see in the sequel.\par
The construction of the parametrix via Gabor frames, having the previous Gaussian beams as building blocks, is performed as follows.
Set
\begin{equation}\label{h}
h= 2\pi \hbar,
\end{equation}
and consider a $\hbar$-Gabor frame $\G^\hbar(\phi_0^\hbar,h^{1/2}\Lambda)$ with $\Lambda=\alpha \bZ^d\times \beta \bZ^d$, $\alpha,\beta>0$. Let $\gamma^h$ be a dual window in $\cS(\rd)$ (see the next section for details). For $N\geq0$, $t\in [0,T]$, the parametrix to \eqref{C1} is defined by
\begin{equation}\label{Gab_exp}
[U^{(N)}(t) f](t,\cdot)=\sum_{z\in h^{1/2}\Lambda}\langle f,\th(z)\gamma^h\rangle \phi_z^{\hbar,N}(t,\cdot).
\end{equation}
Observe that $U^{(N)}(0) f=f$.\par
The following assumptions will be imposed throughout the
paper. \par
\par\medskip
{\bf Assumption (H)}. {\it Suppose that the symbol $H(t,X)$ is continuous with respect to $(t,X)\in [0,T]\times \rdd$ and smooth in $X$, satisfying
\begin{equation}\label{ipotesi}
|\partial^\alpha_X H(t,X)|\leq C_\alpha,\quad \forall |\alpha|\geq 2,\ X\in\rdd,\ t\in [0,T].
\end{equation}
}
This is our main result.
\begin{theorem}\label{mainteo}
Under the {\bf Assumption (H)} and with the above notation, there exists a constant $C=C(T)$ such that, for every $f\in\lrd$,
\begin{equation}\label{teoa0}
\|U^{(N)}(t)f\|_{L^2(\rd)}\leq C\|f\|_{L^2(\rd)}\quad \forall t\in[0,T]
\end{equation}
and
\begin{equation}\label{teob0}
\|(i\hbar \partial_t-\widehat{H(t)})U^{(N)}f\|_{L^2(\rd)}\leq C\hbar^{(N+3)/2}\|f\|_{L^2(\rd)}\quad \forall t\in[0,T].
\end{equation}
If $U(t)$ denotes the exact propagator, for every $f\in\lrd$,
\begin{equation}\label{teoc0}
\|(U^{(N)}(t)-U(t))f\|_{L^2(\rd)}\leq Ct \hbar^{(N+1)/2}\|f\|_{L^2(\rd)} \quad \forall t\in[0,T].
\end{equation}
\end{theorem}

The pioneering papers in this spirit, for Hamiltonians $H(t,x,p)=V(x)+p^2/2$, come back to \cite{hagedorn1980semiclassical,hagedorn1981semiclassical}. More general Hamiltonians were considered in \cite{combescure2012quadratic,robert2004propagation}, which inspired this work.\par
In the framework of nonlinear approximation we can also consider nonlinear parametrices, constructed as follows.\par
Let $\eta\geq0$ be a threshold, and for $f\in L^2(\rd)$ consider the index set
\[
A_{\eta,f}=\{z\in h^{1/2}\Lambda: |\langle f,\th(z)\gamma^h\rangle|>\eta\},
\]
and the nonlinear operator
\[
U^{(N)}_\eta(t)[f](t,\cdot)=\sum_{z\in A_{\eta,f}}\langle f,\th(z)\gamma^h\rangle \phi_z^{\hbar,N}(t,\cdot).
\]
In particular, for $\eta=0$ we recover $U^{(N)}(t)$. \par
In this case we attain the following issue.
\begin{theorem}\label{mainteo1}
Under the {\bf Assumption (H)} and with the above notation, there exists a constant $C=C(T)>0$ such that, for every $f\in L^2(\rd)$, $\eta\geq0$,
\begin{equation}\label{teoa1}
\|U^{(N)}_\eta(t)[f]\|_{L^2(\rd)}\leq C\Big(\sum_{z\in A_{\eta,f}}|\langle f,\th(z)\gamma^h\rangle|^2\Big)^{1/2}\quad \forall t\in[0,T]
\end{equation}
and
\begin{equation}\label{teob1}
\|(i\hbar \partial_t-\widehat{H(t)})U^{(N)}_\eta [f]\|_{L^2(\rd)}\leq C\hbar^{(N+3)/2}\Big(\sum_{z\in A_{\eta,f}}|\langle f,\th(z)\gamma^h\rangle|^2\Big)^{1/2}\quad \forall t\in[0,T].
\end{equation}
If $U(t)$ denotes the exact propagator, for every $f\in L^2(\rd)$,
\begin{multline}\label{teoc1}
\|U^{(N)}_\eta(t)[f]-U(t)f\|_{L^2(\rd)}\leq C \Big(\sum_{z\not\in A_{\eta,f}}|\langle f,\th(z)\gamma^h|^2\Big)^{1/2}\\
+ Ct \hbar^{(N+1)/2}\Big(\sum_{z\in A_{\eta,f}}|\langle f,\th(z)\gamma^h\rangle|^2\Big)^{1/2}\ \forall t\in[0,T].
\end{multline}
\end{theorem}
A similar nonlinear parametrix was constructed in \cite{Lex2} in the case $N=0$. In particular, the estimate \eqref{teob1} already appeared there (Theorem 5.1), but with an additional factor $\sqrt{|A_{\eta,f}|}$ in the righ-hand side:  that estimate blows up when $\eta\to 0$  if $f$ has an infinite number of non zero Gabor coefficients (i.e. $|A_{0,f}|=+\infty$), whereas we see that this is not the case in \eqref{teob1}. In addition, the parametrix in \cite{Lex2} was constructed by means of a truncated Gaussian, which introduces a further error in the estimate.\par
As a byproduct of these techniques, in Section 4 we will also extend the weak deformation of frames result in \cite[Proposition 18]{deGossonACHA} to the case of higher order deformations. However, since the approach is perturbative in nature, our result just holds for $t \hbar^{1/2}$ small enough, and no longer for every $t$ as in \cite{deGossonACHA}.\par
We end up by recalling that Gaussian beam methods have been employed to obtain asymptotic solutions to hyperbolic PDE's in  \cite{Ralston} and hyperbolic systems in \cite{MR3161186,Lex3,lugara2003frame,Lex1,MR2558781,Walden}.

\subsection{Numerical Results} Finally, in Section \ref{sec5} we provide some numerical experiments. We study the Cauchy problem \eqref{C1} for an  Hamiltonian function $H(t,x,p)$ of the form
\begin{equation}\label{gen_Harm_osc}
H(t,x,p) = V(x) + \frac{p^2}{2},
\end{equation}
with an oscillating potential. This is a standard setting for the so-called generalized harmonic oscillator and it is perfectly suited to discuss the behavior of the method in the presence of a potential hill and a potential well.\par
We develop numerical algorithms using MATLAB and the powerful LTFAT\footnote{Large Time Frequency Analysis Toolbox http://ltfat.sourceforge.net/} package, see \cite{sondergaard2007finite,MR2957897}.
We  address the long time propagation of the beams using the reinitialization procedure described in \cite{Lex2}.\par

\section{Preliminaries and  \tfa \,tools}\label{Sect:Prel}
  We refer to  \cite{Gro} for an introduction to time-frequency concepts and in particular to \cite{deGossonbook} for applications to Mathematical Physics. For sake of brevity, sometimes we write
$xy=x\cdot y$,  the scalar product on $\Ren$ and  $|t|^2=t\cdot t$, for $t\in\Ren$.
 The brackets  $\la f,g\ra$
 denote the inner
 product  of $L^2(\Ren)$, i.e. $\la f,g\ra=\int f(t){\overline
 {g(t)}}dt$ and $\|f\|_2^2=\la f,f\ra$.  The Schwartz class is denoted by $\cS(\rd)$.
 For $1\leq p<\infty$, the space $\ell^{p}
(\Lambda )$
 is  the
space of sequences $a=\{{a}_{\lambda}\}_{\lambda \in \Lambda }$
on a  lattice $\Lambda$, such that
$$\|a\|_{\ell^{p}}:=\left(\sum_{\lambda\in\Lambda}
|a_{\lambda}|^p\right)^{1/p}<\infty.
$$
We write ${\rm Sp}(d,\bR)$ for the group of symplectic matrices on $\rdd $, i.e., $\A\in {\rm Sp}(d,\bR)$ if $\A$ is a  $2d\times 2d$ invertible matrix such that $\A ^T J \A  = J$,
where
\begin{equation}
\label{matriceJ}
J=\begin{pmatrix} 0_d&I_d\\-I_d&0_d\end{pmatrix}.
\end{equation}
The standard symplectic form on $\rdd $ is
\begin{equation}
\sigma(z,z')=\;(z')^T J z, \qquad z,z'\in\R^{2d}.
\label{symp}\end{equation}
The metaplectic group is denoted by $Mp(d)$. Consider $\widehat{S}\in Mp(d)$ with covering projection $\pi^\hbar: \widehat{S}\mapsto S \in {\rm Sp}(d,\R)$.  The appearance of the subscript $\hbar$ is due to the fact that to the $\hbar$-dependent operator $\widehat{V}_Pf(x)= e^{-i Px \cdot x/(2\hbar)}f(x)$ (chirp) corresponds the projection $\pi^\hbar (\widehat{V}_P)=V_P$, with $V_P=\begin{pmatrix} I_d&0_d\\-P&0_d\end{pmatrix}$, $P=P^T$,
  and to the Fourier transform $\widehat {J} f(x)=(2\pi i \hbar)^{-d/2}\intrd e^{-i x x'/\hbar} f(x') \,dx'$  corresponds    $\pi^\hbar(\widehat{J})=J$, defined in \eqref{matriceJ}. For details see \cite[Appendix A]{deGossonACHA} and the books \cite{deGossonbook, deGossonbookG}.
 In particular, for $\lambda>0$ we shall use the metaplectic operator $\widehat{M}_{\lambda}\in Mp(d)$  defined (up to a sign) by
\begin{equation}\label{oplambda}\widehat{M}_{\lambda} f(x)=\lambda^{d/2} f(\lambda x),\,\quad f\in\lrd
\end{equation}
 and whose projection is $\pi^\hbar (\widehat{M}_{\lambda})=M_{\lambda}$, the symplectic  matrix
\begin{equation}
\label{matricelambda}
M_{\lambda}=\begin{pmatrix} \lambda^{-1} I_d& 0_d\\0_d&\lambda I_d.\end{pmatrix}
\end{equation}
In the sequel we shall often use the fundamental symplectic covariance formula
\begin{equation}\label{CF}
\tauhz \widehat{S}= \widehat{S}\th(S^{-1}z)\quad S\in {\rm Sp}(d,\R).
\end{equation}

\subsection{$\hbar$-Gabor Frames}
 The definition of a $\hbar$-Gabor frame is already contained in the introduction. Consider a lattice $\Lambda$ in $\rdd$. The Gabor system $$\mathcal{G}(\f,\Lambda) =\{\widehat{\mathcal{T}}(z)\f,z\in\Lambda\},$$
(recall that $\t(z)=\t^{(2\pi)^{-1}}(z)$) is a Gabor frame for $\lrd$ if there exist
constants $a,b> 0 $ such that for every $f\in\lrd$
\begin{equation}
\label{Frame_rel}
a\|f\|^2_2 \leq \sum_{z\in\Lambda} |\la f,\widehat{\mathcal{T}}(z)\f|^2 \leq b\|f\|^2_2.
\end{equation}
If $\eqref{Frame_rel}$ holds,
then there exists a  $\gamma\in\lrd$ (so-called dual window),
such that $\mathcal{G}(\gamma,\Lambda)$
is a frame for $\lrd$ and every $f\in\lrd$ can be expanded as
\begin{equation}\label{GabExp}
f=\sum_{z\in\Lambda}\la f, \widehat{\mathcal{T}}(z)\f \ra \widehat{\mathcal{T}}(z)\gamma= \sum_{z\in\Lambda}\la f,\widehat{\mathcal{T}}(z)\gamma\ra \widehat{\mathcal{T}}(z)\f,
\end{equation}
with unconditional convergence in $\lrd$. In particular,  if the window  $\f$  is a Gaussian function, then there exists a dual window $\gamma$ that is smooth and well-localized, in particular  $\gamma\in\cS(\rd)$ (see \cite{GaborEvolEq,grochenig2009gabor}).\par
In what follows we investigate some useful properties which let us switch from a $\hbar$-Gabor frame to a standard Gabor frame and vice-versa. To reach this goal, we define the dilation matrix $D^{\hbar}$ and its inverse $(D^{\hbar})^{-1}$
\begin{equation}\label{Dh}
D^{\hbar}=\begin{pmatrix} I_d&0_d\\0_d& h I_d.\end{pmatrix}\quad (D^{\hbar})^{-1}=\begin{pmatrix} I_d&0_d\\0_d&h^{-1} I_d\end{pmatrix}
\end{equation}
(recall that $h=2\pi \hbar$).
\begin{proposition}\label{P1}
Let $D^{\hbar}$ be the matrix defined in \eqref{Dh}. The system $\Ghframe$ is a $\hbar$-Gabor frame if and only if $\mathcal{G}(\f,(D^{\hbar})^{-1}\Lambda)$ is a Gabor frame and the frame bounds are the same. Moreover, every dual window $\gamma$ of the Gabor frame $\mathcal{G}(\f,(D^{\hbar})^{-1}\Lambda)$ originates a $\hbar$-Gabor frame $\mathcal{G}^\hbar(\gamma,\Lambda)$, dual frame of $\Ghframe$.
\end{proposition}
\begin{proof}
The first part is straightforward and follows the pattern of \cite[Prop. 7]{deGossonACHA}. Precisely, the system $\Ghframe$  is a $\hbar$-Gabor frame if and only if there exist positive constants $a,b>0$ such that  \eqref{framedef} holds. Setting $z=(x,2\pi \hbar p)\in\Lambda$ if and only if $(x,p)\in (D^\hbar)^{-1}\Lambda$, and using $\widehat{\mathcal{T}}^\hbar(x,2\pi \hbar p)=\widehat{\mathcal{T}}(x,p)$, the inequalities \eqref{framedef} are equivalent to
\begin{equation*}
a\|f\|_2^2\leq\sum_{z\in(D^\hbar)^{-1}\Lambda}|\la f,\widehat{\mathcal{T}}\f\ra|^2\leq b\|f\|_2^2,\quad \forall f \in\lrd,
\end{equation*}
as desired. \par
Now, consider a dual window $\gamma\in\lrd$ of the Gabor frame $\mathcal{G}(\f,(D^{\hbar})^{-1}\Lambda)$.  Then every $f\in\lrd$ can be expanded as
\begin{align*}
f&=\sum_{z\in(D^{\hbar})^{-1}\Lambda}\la f, \widehat{\mathcal{T}}(z)\f \ra \widehat{\mathcal{T}}(z)\gamma \\
&=\sum_{(x,p)\in(D^{\hbar})^{-1}\Lambda}\la f, \widehat{\mathcal{T}}^\hbar(x,2\pi \hbar p)\f \ra \widehat{\mathcal{T}}^\hbar(x,2\pi \hbar p)\gamma \\
&=\sum_{(x,p')\in\Lambda}\la f, \widehat{\mathcal{T}}^\hbar(x, p')\f \ra \widehat{\mathcal{T}}^\hbar(x, p')\gamma,
\end{align*}
that is the system $\{\widehat{\mathcal{T}}^\hbar(z),\,z\in\Lambda\}$ is a dual frame of the $\hbar$-Gabor frame $\Ghframe$.

\end{proof}

Given the Gaussian window $\phi_0$ defined in \eqref{Gaussian} and the lattice $\Lambda=\a\zd\times\beta\zd$, with $\a,\beta>0$, de Gosson in \cite[Prop. 12]{deGossonACHA} shows that the system $\mathcal{G}^\hbar(\phi_0^\hbar,h^{1/2}\,\Lambda)$ is a $\hbar$-Gabor frame if and only if $\alpha\beta<1$, this means by the previous proposition  that the system \begin{equation}\label{GFh}\mathcal{G}(\phi_0^\hbar,\a h^{1/2}\,\zd\times \beta h^{-1/2}\zd)\end{equation} is a Gabor frame if and only if $\alpha\beta<1$.  This frame will be used for the numerical experiments in Section \ref{sec5}.\par
Using the same arguments as in the proof of Proposition \ref{P1} we obtain the following characterization:
\begin{proposition}\label{P2}
Let $D^{\hbar}$ be the matrix defined in \eqref{Dh} and consider a Gabor frame $\mathcal{G}(\f,\Lambda)$.  The system $\mathcal{G}(\gamma,\Lambda)$ is a dual Gabor frame  of $\mathcal{G}(\f,\Lambda)$ if and only if  $\mathcal{G}^\hbar(\gamma,D^{\hbar}\Lambda)$    is a dual   $\hbar$-Gabor frame of the $\hbar$-Gabor frame $\mathcal{G}^\hbar(\f,D^{\hbar}\Lambda)$. Moreover, the frame bounds of $\mathcal{G}(\f,\Lambda)$ and $\mathcal{G}^\hbar(\f,D^{\hbar}\Lambda)$ are the same.
\end{proposition}
We shall work with $\hbar$-Gabor frames where both  windows and lattices are  rescaled. Their dual frames behave as follows. We set
\begin{equation}\label{fh}
\f^h(x)=h^{-d/4}\f(h^{-1/2}x)=\widehat{M}_{h^{-1/2}}\f(x)=\widehat{M}_{(2\pi \hbar)^{-1/2}}\f(x).
\end{equation}
\begin{proposition}\label{P3}
Consider a Gabor frame $\mathcal{G}(\f,\Lambda)$.  The system $\mathcal{G}(\gamma,\Lambda)$ is a dual Gabor frame  of $\mathcal{G}(\f,\Lambda)$ if and only if  $\mathcal{G}^\hbar(\gamma^h,h^{1/2}\,\Lambda)$    is a dual   $\hbar$-Gabor frame of the $\hbar$-Gabor frame $\mathcal{G}^\hbar(\f^h,h^{1/2}\,\Lambda)$. Moreover, the frame bounds of $\mathcal{G}(\f,\Lambda)$ and $\mathcal{G}^\hbar(\f^h,h^{1/2}\,\Lambda)$ are the same.
\end{proposition}
\begin{proof} Consider the metaplectic operator $\widehat{M}_{h^{-1/2}}\in Mp(d)$  and its inverse $\widehat{M}_{h^{1/2}}\in Mp(d)$.
For every $f\in\lrd$, we have $\widehat{M}_{h^{-1/2}}\widehat{M}_{h^{1/2}} f=f$. Given the Gabor frame $\mathcal{G}(\f,\Lambda)$ with dual Gabor frame $\mathcal{G}(\gamma,\Lambda)$, by Proposition \ref{P2} and using the boundedness of the metaplectic operators on $\lrd$ we can write, for every $f\in\lrd$,
\begin{align*}
f&=\widehat{M}_{h^{-1/2}}\sum_{z\in D^{\hbar}\Lambda}\la  \widehat{M}_{h^{1/2}} f, \widehat{\mathcal{T}}^\hbar(z)\f \ra \widehat{\mathcal{T}}^\hbar(z)\gamma \\
&=\widehat{M}_{h^{-1/2}}\sum_{z\in h^{1/2}\,\Lambda}\la  \widehat{M}_{h^{1/2}} f, \widehat{\mathcal{T}}^\hbar( M_{h^{1/2}}\,z)\f \ra \widehat{\mathcal{T}}^\hbar (M_{h^{1/2}}\,z)\gamma \\
&=\sum_{z\in h^{1/2} \,\Lambda}\la f, \widehat{M}_{h^{-1/2}}    \widehat{\mathcal{T}}^\hbar( M_{h^{1/2}}\,z)\f \ra \widehat{M}_{h^{-1/2}}    \widehat{\mathcal{T}}^\hbar( M_{h^{1/2}}\,z)\gamma,\\
&=\sum_{z\in h^{1/2} \,\Lambda}\la f, \widehat{\mathcal{T}}^\hbar(z)\f^h \ra \widehat{\mathcal{T}}^\hbar(z)\gamma^h,
\end{align*}
where we used  $\tauhz \widehat{M}_{h^{-1/2}} = \widehat{M}_{h^{-1/2}} \widehat{\mathcal{T}}^\hbar ( M_{h^{1/2}}\, z)$, by the covariance formula \eqref{CF} and the definition \eqref{fh}. This ends the proof of the equivalence of the dual frames. The proof of the frame bounds follows the argument of the first part of the proof of Proposition \ref{P1}.
\end{proof}

\section{Bounds for the parametrix in $L^2$}
\subsection{Preliminary remarks}
We assume for the Hamiltonian $H(t,X)$ the  validity of   {\bf Assumption (H)}.
In particular, we  consider its second order Taylor term $H^{(2)}_{z_0}$ at $z_0=(x_0,p_0)$, as in \eqref{Hquad} and the corresponding propagator $\st(z_0)$ in \eqref{metap}. We can also consider the operator $\sht(z_0)$ defined by
\begin{equation}\label{metaph}
 i \hbar \partial_t \sht(z_0) =\widehat{H^{(2)}_{z_0}(t)} \sht(z_0)\quad\widehat{S_0}^\hbar(z_0) =I.
 \end{equation}
This operator is related to $\st$ via the formula (omitting the dependence on $z_0$ for simplicity)
 \[
\sht=\mhm\st\mhp.
\]
 Indeed we have
\begin{align*}
i\hbar \sht&=i\hbar \partial_t \mhm\st\mhp\\
&=\hbar \mhm {\rm Op}^w_1[H^{(2)}(t,x,p)] \st \mhp\\
&=\hbar {\rm Op}^w_1[H^{(2)}(t,\hbar^{-1/2}x,\hbar^{1/2}p)]\mhm\st\mhp\\
&={\rm Op}^w_1[H^{(2)}(t,x,\hbar p)]\sht\\
&=\widehat{H^{(2)}(t)}\sht.
\end{align*}
\begin{remark} The action of $\st(z_0)$ or $\sht(z_0)$ on a mudulated Gaussian function can be written down explicitly, for details we refer to Section \ref{sec5} and the references quoted there.
\end{remark}
\begin{remark}\label{remark3.2}
The projection $S_t(z_0)$ represents the flow of the linear system with Hamiltonian $H^{(2)}_{z_0}(t,X)$. As a matrix, $S_t(z_0)$ is in ${\rm Sp}(d,\R)$, for every $t\in [0,T]$. The entries of the matrix $S_t(z_0)$  depend on $t\in[0,T]$ and $z_0\in\rdd$ but they are bounded, because this is true for the coefficients of the polynomial $H^{(2)}_{z_0}(t,X)$ by the assumption \eqref{ipotesi}. The same is true for the entries of the inverse matrix $S_t^{-1}(z_0)$.
\end{remark}

\subsection{Evolution of a coherent state}
Consider the Cauchy problem \eqref{C1z0}. Let $z_t$ be the trajectory of the corresponding Hamiltonian system, with initial condition $z_0$. We will show that an approximate solution (Gaussian beam) is given by
\begin{align}\label{TE}
\phi^{\hbar,0}_{z_0}(t,\cdot)&=e^{\frac{i}{\hbar}\delta(t,z_0)} \th(z_t) \sht(z_0)\phi^\hbar_0\\
&=e^{\frac{i}{\hbar}\delta(t,z_0)} \th(z_t) \mhm\st(z_0)\mhp\phi^\hbar_0\notag\\
&=e^{\frac{i}{\hbar}\delta(t,z_0)} \th(z_t) \mhm\st(z_0)\phi_0,\notag
\end{align}
where $\delta(t,z_0)$ is the symmetrized action defined in \eqref{symmaction}. More generally, we will consider higher order approximations $\phi^{\hbar,N}_{z_0}(t)$ in the form \eqref{solapprox}.\par
To this end we will have to estimate the remainder term
\begin{equation}\label{resto}
R^{(N)}_{z_0}(t,\cdot):=(i\hbar \partial_t -\widehat{H(t)})\phi^{\hbar,N}_{z_0}(t,\cdot).
\end{equation}
The following computation of $R^{(N)}_{z_0}$ were carried out in \cite{hagedorn1980semiclassical,hagedorn1981semiclassical} for Hamiltonians of the form $H(t,x,p)=V(x)+p^2/2$ and in \cite{combescure2012quadratic,robert2004propagation} for more general Hamiltonians with polynomial growth. Here we briefly sketch the main points for the benefit of the reader, because the formula given in \cite[(70)]{robert2004propagation} for $R^{(N)}_{z_0}$ contains a number of misprints. An explicit computation show that
\begin{align}\label{uno}
i\hbar \partial_t e^{\frac{i}{\hbar}\delta(t,z_0)}
&=-e^{\frac{i}{\hbar}\delta(t,z_0)}\Big(\frac{1}{2}\sigma(z_t,\dot{z}_t)-H(t,z_t)\Big)\\
&=- e^{\frac{i}{\hbar}\delta(t,z_0)}\Big(\frac{1}{2}(p_t\dot{x_t}-\dot{p_t}x_t)-H(t,z_t)\Big). \nonumber
\end{align}
On the other hand, given a function $f(x)$ we have
\begin{align}
i\hbar \partial_t\th(z_t) f(x)&=
e^{\frac{i}{\hbar}(p_t x-x_tp_t/2)}
\Big((-\dot{p_t} x+\frac{\dot{p_t}x_t+p_t\dot{x_t}}{2})f(x-x_t)-i\hbar\dot{x_t} \nabla f(x-x_t)\Big)\\
&=\th(z_t)\Big(-\dot{p_t}x-\dot{p_t}x_t+\frac{\dot{p_t}x_t+p_t\dot{x_t}}{2}-i\hbar\dot{x_t}\nabla  \Big) f(x)\nonumber\\
&=\th(z_t)\Big(\frac{p_t\dot{x_t}-\dot{p_t}x_t}{2}+x \partial_x H(t,z_t)+\hbar \partial_p H(t,z_t)(-i\nabla)\Big)f(x),\nonumber
\end{align}
which implies
\begin{multline}\label{due}
i\hbar \partial_t\th(z_t) \mhm\\
=\th(z_t)\mhm\Big[i\hbar\partial_t+\frac{p_t\dot{x_t}-\dot{p_t}x_t}{2}+\hbar^{1/2} x \partial_x H(t,z_t)+\hbar^{1/2}\partial_p H (t,z_t)(-i\nabla)\Big].
\end{multline}
We also have the covariance formula
\[
\widehat{H(t)}\th(z_t)=\th(z_t){\rm Op}^w_1[H(t,x+x_t,\hbar(p+\hbar^{-1}p_t)]
\]
which implies
\begin{equation}\label{tre}
\widehat{H(t)}\th(z_t)\mhm=\th(z_t)\mhm{\rm Op}^w_1[H(t, h^{1/2}x+x_t,\hbar^{1/2}p+p_t)].
\end{equation}
Finally a Taylor expansion yields
\begin{multline}\label{quattro}
H(t, \hbar^{1/2}x+x_t,\hbar^{1/2}p+p_t)=H(t,z_t)+\hbar^{1/2} x\partial_x H(t,z_t)+\hbar^{1/2} p\partial_p H(t,z_t)\\
+\hbar H^{(2)}_{z_0}(t,X)
+\sum_{l=3}^{N+2} \hbar^{l/2} H^{(l)}_{z_0}(t,X)+\hbar^{(N+3)/2} r^{(N+3)}_{z_0}(t,X),
\end{multline}
where
\[
H^{(l)}_{z_0}(t,X)=\sum_{|\gamma|=l}\frac1 {\gamma!}\partial^\gamma_X H(t,z_t)X^\gamma,\quad X\in\rdd
\]
and
\[
r^{(N+3)}_{z_0}(t,X)=\frac{1}{(N+3)!}\sum_{|\gamma|=N+3} \int_0^1 \partial^\gamma_X H(t,z_t+\theta\hbar^{1/2} X) X^\gamma (1-\theta)^{N+2}\, d\theta.
\]
By the formulas \eqref{uno}--\eqref{quattro} we obtain
\begin{align*}
&(i\hbar\partial_t -\widehat{H(t)})\phi^{\hbar,N}_{z_0}(t,\cdot)=e^{\frac{i}{\hbar}\delta(t,z_0)}\th(z_t)\mhm\Big(i\hbar\partial_t-\hbar {\rm Op}^w_1[H^{(2)}_{z_0}(t,X)]\\
&-\sum_{l=3}^{N+2}\hbar^{l/2}{\rm Op}^w_1[H^{(l)}_{z_0}(t,X)]-\hbar^{(N+3)/2}{\rm Op}^w_1[r^{(N+3)}_{z_0}(t,X)])\Big)\st(z_0)\sum_{j=0}^N \hbar^{j/2}b_j(t,\cdot)\phi_0\\
&=e^{\frac{i}{\hbar}\delta(t,z_0)}\th(z_t)\mhm\st(z_0)\Big(i\hbar\partial_t
-\sum_{l=3}^{N+2}\hbar^{l/2}{\rm Op}^w_1[H^{(l)}_{z_0}(t,S_t(X))]\\
&\qquad\qquad\qquad\qquad\qquad\quad-\hbar^{(N+3)/2}{\rm Op}^w_1[r^{(N+3)}_{z_0}(t,S_t(X))]\Big)\sum_{j=0}^N \hbar^{j/2}b_j(t,\cdot)\phi_0.
\end{align*}
Now we choose $b_0(t,X)=1$ and $b_j(t,X)$, $j\geq 1$ solutions to
\[
\begin{cases}
i\partial_t b_j(t,X)\phi_0(x)=\sum_{l+k=j+2\atop l\geq 3}{\rm Op}^w_1[H^{(l)}_{z_0}(t,S_t(X))](b_k(t,\cdot)\phi_0)(x)\\
b_j(t,X)=0.
\end{cases}
\]
\begin{remark}
Since $\phi_0$ is a Gaussian function and ${\rm Op}^w_1[H^{(l)}_{z_0}(t,S_t(X))]$ are differential operators with polynomial coefficients depending on $t,z_0$, we see that $b_j(t,X)$ is a polynomial in $X$, having coefficients depending on $t,z_0$ which are bounded, for the entries of the matrix $S_t$,  as functions of $t\in [0,T], z_0\in\rdd$, are bounded, as well as the coefficients of the polynomial $H^{(l)}_{z_0}(t,X)$, by the assumption \eqref{ipotesi}.
\end{remark}
With this choice of $b_j(t,X)$ we finally obtain the desired formula for $R^{(N)}_{z_0}$:
\begin{multline}\label{cinque}
R^{(N)}_{z_0}(t,\cdot)=(i\hbar \partial_t -\widehat{H(t)})\phi^{\hbar,N}_{z_0}(t,\cdot)\\
=-e^{\frac{i}{\hbar}\delta(t,z_0)}\th(z_t)\mhm\st(z_0) \Big(
\sum_{l+k\geq N+3\atop {3\leq l\leq N+2 \atop 0\leq k\leq N}}\hbar^{(l+k)/2}{\rm Op}^w_1[H^{(l)}_{z_0}(t,S_t(X)) ]b_k(t,\cdot)\phi_0\\
+\sum_{k=0}^N \hbar^{(N+3+k)/2}{\rm Op}^w_1[r^{(N+3)}_{z_0}(t,S_t(X))]b_k(t,\cdot)\phi_0\Big).
\end{multline}
\subsection{Bounds for the parametrix: proof of Theorems \ref{mainteo}, \ref{mainteo1}}
From now on we work with the lattice $\Lambda=\alpha\zd \times\beta\zd$, $\alpha,\beta>0$.
We shall need the preliminary estimate below.
\begin{theorem}\label{mainteo2}
Let $\G^\hbar(\phi_0^\hbar,h^{1/2}\Lambda)$ be a $\hbar$-Gabor frame. There exists a constant $C=C(T)$ such that, for every sequence
 $\{c_z:\,z\in h^{1/2}\Lambda\}\in \ell^2(h^{1/2}\Lambda)$ we have
\begin{equation}\label{teoa}
\|\sum_{z\in h^{1/2}\Lambda}c_z \phi^{\hbar,N}_z(t,\cdot)\|_{L^2(\rd)}\leq C \Big(\sum_{z\in h^{1/2}\Lambda}|c_z|^2\Big)^{1/2}\quad \forall t\in[0,T]
\end{equation}
and
\begin{equation}\label{teob}
\|\sum_{z\in h^{1/2}\Lambda}c_z (i\hbar \partial_t-\widehat{H(t)}) \phi^{\hbar,N}_z(t,\cdot)\|_{L^2(\rd)}\leq C\hbar^{(N+3)/2}\Big(\sum_{z\in h^{1/2}\Lambda}|c_z|^2\Big)^{1/2}\quad \forall t\in[0,T].
\end{equation}
\end{theorem}
\begin{proof}
Let us prove \eqref{teoa}. We have
\[
\|\sum_{z\in h^{1/2}\Lambda} c_z\phi^{\hbar,N}_z(t,\cdot) \|^2_{L^2(\rd)}=\sum_{z,z'\in h^{1/2}\Lambda} c_z\overline{c_{z'}}\langle \phi_z^{\hbar,N}(t,\cdot), \phi_{z'}^{\hbar,N}(t,\cdot)\rangle.
\]
We will prove that
\begin{equation}\label{sei}
|\langle \phi_z^{\hbar,N}(t,\cdot), \phi_{z'}^{\hbar,N}(t,\cdot)\rangle|\leq G(z-z'),
\end{equation}
where
$G\in \ell^1(h^{1/2}\Lambda)$, with $\ell^1$-norm independent of $\hbar$.
By the Cauchy-Schwarz inequality in $\ell^2(h^{1/2}\Lambda)$ and  Young inequality we have
\begin{align*}
\|\sum_{z\in h^{1/2}\Lambda} c_z\phi^\hbar_z(t,\cdot) \|^2_{L^2(\rd)}&\leq \Big(\sum_{z\in h^{1/2}\Lambda}|c_z|^2\Big)^{1/2}\Big(\sum_{z\in h^{1/2} \Lambda} |\sum_{z'\in h^{1/2}\Lambda} G(z-z')\overline{c_{z'}}|^2\Big)^{1/2}\\
&\leq \|G\|_{\ell^1(h^{1/2}\Lambda)}\sum_{z\in  h^{1/2}\Lambda}|c_z|^2.
\end{align*}
It remains to prove \eqref{sei}.
If we write down explicitly the expression of the functions $\phi_z^{\hbar,N}(t,\cdot)$, $\phi_{z'}^{\hbar,N}(t,\cdot)$, we see that it is sufficient to prove that
\begin{equation}\label{sette}
|\langle\th(z_t)\mhm\st(z)g_1, \th(z'_t)\mhm\st(z')g_2\rangle|\leq G(z-z'),
\end{equation}
with $G\in \ell^1( h^{1/2}\Lambda)$ as above, when $g_1,g_2$ are Schwartz functions. \par
If we denote by $I_{z,z'}$ the left-hand side of \eqref{sette} we have, with $z_t=(x_t,p_t)$, $z'_t=(x'_t,p'_t)$,
\begin{align*}
I_{z,z'}&=|\langle \th(z_t-z'_t) \mhm \st(z) g_1,\mhm\st(z')g_2 \rangle|\\
&=  |\langle \th(\hbar^{-1/2}(x_t-x'_t),\hbar^{1/2}(p_t-p'_t))\st(z)g_1,\st(z')g_2 \rangle| \\
&= |\langle \t^1(\hbar^{-1/2}(z_t-z'_t))\st(z)g_1,\st(z')g_2 \rangle|.
\end{align*}
Here $\t^1$ stands for $\th$ with $\hbar=1$.\par
Now, metaplectic operators are bounded $\cS(\rd)\to\cS(\rd)$ \cite{deGossonbook} and, as already observed in Remark \ref{remark3.2}, the entries of the matrix $S_t(z)$ are bounded as functions of $t\in[0,T]$ and $z\in\rdd$. Hence the functions $\st(z)g_1$ and $\st(z')g_2$ are Schwartz functions with Schwartz seminorms bounded with respect to $t,z,z'$.\par Since the map (essentially the Short-time Fourier transform in Time-frequency Analysis)
\[
(f_1,f_2)\mapsto \langle\t^1(z)f_1,f_2\rangle
\]
is continuous $\cS(\rd)\times\cS(\rd)\to\cS(\rdd)$ \cite{Gro}, there exists a constant $C=C(T)$ such that
\begin{align*}
I_{z,z'}&\leq C(1+\hbar^{-1/2}|z_t-z'_t|)^{-(2d+1)}\\
&\leq C(1+ h^{-1/2}|z-z'|)^{-(2d+1)}
\end{align*}
where in the last inequality we used the fact that the maps $z\to z_t$ has an inverse which is globally Lipschitz continuous. This easily follows from the assumption \eqref{ipotesi} (see e.g. \cite{fio3}).
Hence \eqref{sei} is verified with $G(z)=(1+h^{-1/2}|z|)^{-(2d+1)}$, and
\[
\|G\|_{\ell^1( h^{1/2}\Lambda)}=\sum_{\tilde{z}\in \Lambda} (1+|\tilde{z}|)^{-(2d+1)}=C'<+\infty,
\]
and $C'$ is of course independent of $\hbar$.\par
Let us now prove \eqref{teob}. The proof is similar to that of \eqref{teoa}. Indeed, by arguing as above and using the expression for $R^{(N)}_z(t,x)$  in \eqref{cinque} we see first of all that we gain a further factor $\hbar^{(N+3)/2}$. Moreover we are again reduced to estimate terms of the same form as in \eqref{sette}, where now the functions $g_1$ and $g_2$ depend on $t\in[0,T],\ z,z'\in\rdd$. The point is that their seminorms in the Schwartz space remain bounded. This follows from the fact that in \eqref{cinque} the functions $b_j(t,\cdot)\phi_0$ have seminorms in the Schwartz space bounded with respect to $t\in[0,T]$ and $z_0\in\rdd$, and the pseudodifferential operators which act on them have symbols in the H\"ormander class $S^0_{0,0}$ (i.e.\ bounded together with all their derivatives) with seminorms uniformly bounded with respect to $t,z_0$ (which in turns is a consequence of Remark \ref{remark3.2}).  This concludes the proof of \eqref{teob}.\par
\end{proof}

\begin{proof}[Proof of Theorem \ref{mainteo}]
The bounds \eqref{teoa0} and \eqref{teob0} follow at once from \eqref{teoa} and \eqref{teob} with
\[
c_z=\langle f,\th(z)\gamma^h\rangle.
\]
Indeed, we have
\[
f=\sum_{z\in  h^{1/2}\Lambda}\langle f,\th(z)\gamma^h\rangle \phi_z^\hbar
\]
and
\[
\Big(\sum_{z\in  h^{1/2}\Lambda} |c_z|^2\Big)^{1/2}\leq C \|f\|_{L^2},
\]
for a constant $C>0$ independent of $\hbar$, by Proposition \ref{P3}.\par
Finally we prove \eqref{teoc0}. This follows from Duhamel formula: if $U(t,s)$ is the exact propagator, with $U(s,s)=I$, and $U(t)=U(t,0)$, we have
\[
U^{(N)}(t)f-U(t)f=-\frac{i}{\hbar}\int_0^t U(t,s)\Big(i\hbar\partial_s-\widehat{H(s)} \Big)U^{(N)}(s)f\,ds,
\]
and \eqref{teoc0} follows from \eqref{teob0} and the the fact that $U(t,s)$ is a unitary operator in $L^2(\rd)$.
\end{proof}
\begin{proof}[Proof of Theorem \ref{mainteo1}]
The estimates \eqref{teoa1}, \eqref{teob1} follow from Theorem \ref{mainteo2}.\par
Concerning \eqref{teoc1} we see that now Duhamel's formula reads
\[
U^{(N)}_\eta(t)[f]-U(t)f=U^{(N)}_\eta(0)[f]-U(0)f-\frac{i}{\hbar}\int_0^t U(t,s)\Big(i\hbar\partial_s-\widehat{H(s)} \Big)U^{(N)}_\eta(s)[f]\,ds,
\]
which gives the desired result, using \eqref{teob1} and the representation
\[
U^{(N)}_\eta(0)[f]-U(0)f=\sum_{z\not\in A_{\eta,f}} \langle f,\th(z)\gamma^h\rangle\phi^\hbar_z.
\]
This concludes the proof.
\end{proof}
\section{Higher order deformation of frames}
With the above notation, consider the function
\[
\phi^{\hbar,N}_0(t,x)= e^{\frac{i}{\hbar}\delta(t,0)}\th(z_t)\sht(0)\sum_{j=0}^N \hbar^{j/2} b_j(t,\hbar^{-1/2} x)\phi_0^\hbar(x).
\]
Here $z_t$ is the integral curve of the Hamiltonian system \eqref{HS} with initial condition $z_0=0$. Also, the $b_j$'s are constructed as in Section 3.2 by considering as initial value the Gaussian $\phi^\hbar_0$, which is centered at $z_0=0$.
Let $f_t$ be the Hamiltonian flow defined by $H(t,X)$. In particular $z_t=f_t(0)$.\par
The following result extends \cite[Proposition 18]{deGossonACHA}; it reduces to that result for $N=0$, at least for $t\hbar^{1/2}$ small enough.
\begin{theorem}\label{teodeformazione}
Assume the validity of {\bf Assumption (H)} and consider a $\hbar$-Gabor frame $\mathcal{G}^\hbar(\phi_0^\hbar,  h^{1/2}\Lambda)$. Then there exists a constant $\epsilon>0$, depending only on $H$, the frame bounds of $\mathcal{G}^\hbar(\phi_0^\hbar, h^{1/2}\Lambda)$ (which are independent of $\hbar$ by Proposition \ref{P3}) and $\Lambda$, such that for $t\in [0,T]$, $t\hbar^{1/2}<\epsilon$, the system $\mathcal{G}^\hbar(\phi_0^{\hbar,N}(t,\cdot),f_t( h^{1/2}\Lambda))$ is a $\hbar$-Gabor frame, with frame bounds independent of $\hbar$.
\end{theorem}

\begin{proof}
Let
\[
\tilde{\phi}^{\hbar,N}_0(t,x)=\sum_{j=0}^N \hbar^{j/2} b_j(t,\hbar^{-1/2} x)\phi_0^\hbar(x),
\]
so that
\[
{\phi}^{\hbar,N}_0(t,\cdot)= e^{\frac{i}{\hbar}\delta(t,0)}\th(z_t)\sht(0)\tilde{\phi}^{\hbar,N}_0(t,\cdot).
\]
By arguing exactly as in \cite[Proposition 18]{deGossonACHA}, we see that it suffices to prove that $\mathcal{G}^\hbar(\tilde{\phi}_0^{\hbar,N}(t,\cdot), h^{1/2}\Lambda)$ is a Gabor frame, at least for $t\hbar^{1/2}$ small enough. By assumption we have
\[
a\|f\|^2_{L^2(\rd)}\leq\sum_{z\in  h^{1/2}\Lambda}|\langle f, \th(z) \phi^\hbar_0\rangle|^2\leq b\|f\|_{L^2(\rd)}
\]
for some $0<a\leq b$ independent of $\hbar$.\par
Set
\[
\psi_0^{\hbar,N} (t,x):=\tilde{\phi}^{\hbar,N}_0(t,x)-\phi^\hbar_0(x)=\sum_{j=1}^N \hbar^{j/2}b_j(t,\hbar^{-1/2} x)\phi_0^\hbar (x)
\]
(we used the fact that $b_0(t,x)=1$). By the triangle inequality it is sufficient to prove that
\[
\sum_{z\in h^{1/2}\Lambda}|\langle f,\th(z) \psi_0^{\hbar,N}(t,\cdot)|^2\leq \frac{a}{2}\|f\|^2_{L^2(\rd)}.
\]
 Using the representation
\[
f=\sum_{z\in  h^{1/2}\Lambda} \langle f, \th(z)\gamma^h\rangle \th(z)\phi^\hbar_0
\]
with $\gamma^h$ being a dual window in $\cS(\rd)$, and using Young inequality in $\ell^2( h^{1/2}\Lambda)$ we are reduced to prove that
\[
I_{z,z'}:=|\langle \th(z)\phi_0^\hbar,\th(z')\psi^{\hbar,N}_0(t,\cdot) \rangle|\leq G(z-z')
\]
for some $G\in \ell^1( h^{1/2}\Lambda)$ with $\|G\|^2_{\ell^1}\leq a/(2b)$.\
\par
Now, setting
\[
\psi_0(t,x):=\sum_{j=1}^N \hbar^{j/2}b_j(t,x)\phi_0(x)=\hbar^{d/4}\psi^{\hbar,N}_0(t,\hbar^{1/2}x)
\]
and $z=(x_0,p_0)$, $z'=(x_0',p_0')$, we have
\begin{align*}
I_{z,z'}&=|\langle \th(z-z')\phi^\hbar_0,\psi^{\hbar,N}_0\rangle|\\
&=\hbar^{-d/2}\Big|\int e^{i\frac{p_0-p'_0}{\hbar}x}\phi_0(\hbar^{-1/2}(x-(x_0-x'_0))\big)\overline{\psi_0(\hbar^{-1/2}x)}\, dx\Big|\\
&=\big|\langle \t^1(\hbar^{-1/2}(z-z'))\phi_0,\psi_0)\rangle\big|.
\end{align*}
Observe that $\psi_0$ is a Schwartz function whose seminorms are dominated by $t\hbar^{-1/2}$, because $b_j(0,x)=0$ for $j\geq1$.
Hence
\[
I_{z,z'}\leq Ct\hbar^{1/2}(1+h^{-1/2}|z-z'|)^{-(2d+1)},
\]
for some constant $C=C(T)>0$. If we define $G(z)=Ct\hbar^{1/2}(1+h^{-1/2}|z|)^{-(2d+1)},$ we have
\[
\sum_{z\in  h^{1/2}\Lambda} G(z)=\sum_{\tilde{z}\in \Lambda} Ct\hbar^{-1/2}(1+|\tilde{z}|)^{-(2d+1)}\leq C't\hbar^{1/2}.
\]
The desired conclusion then follows if we choose $(C't\hbar^{1/2})^2\leq a/(2b)$.
\end{proof}

\section{Numerical Results}\label{sec5}
The aim of this section is to construct a parametrix of order $N = 0$  for an  Hamiltonian function $H(t,x,p)$ of the form \eqref{gen_Harm_osc},
where $$V\in\cC^{\infty}(\rd)\quad \mbox{with}\quad  |\partial_x^{\al}V(x)| \leq C_{\al}, \quad \forall |\alpha|\geq 2,\ x\in\rd,$$ so that {\bf Assumption (H)} (cf. \eqref{ipotesi}) is satisfied.\par
\subsection{Time Evolution of Gaussian Beams} In what follows we first recall the set of ordinary differential equations controlling the time evolution of the Gaussian beam for the quadratic time dependent Hamiltonian $H^{(2)}_{z_0}(t,X)$ defined in \eqref{Hquad}. The details and the proofs of what follows can be found in many works, we refer for instance  to \cite[Chapter 3]{ComberscureRobert2012}.
Let us represent the quadratic form $H^{(2)}_{z_0}(t,X)$, $X=(x,p)$, in \eqref{Hquad} as
$$H^{(2)}_{z_0}(t,X)=\frac 12 (K_t x\cdot x+2 L_t x\cdot p+G_t p\cdot p)
$$
where $K_t,L_t,G_t$ are real $d\times d$ matrices, $G_t,K_t$ being symmetric.
Setting
$$F(t)=\begin{pmatrix}K_t&L^T_t\\L_t&G_t\end{pmatrix},$$
where $L^T_t$ is the transpose matrix of $L_t$, the classical motion driven by the Hamiltonian $H^{(2)}_{z_0}$ is given by the Hamilton equation
\begin{equation}\label{He} \begin{pmatrix}\dot{x}\\ \dot{p}\end{pmatrix}=JF(t)\begin{pmatrix}{x}\\ {p}\end{pmatrix}.
\end{equation}
The classical flow $S_t(z_0)= \begin{pmatrix}A_t&B_t\\C_t&D_t\end{pmatrix}\in \rm{Sp}(d,\R)$ for the Hamiltonian $H^{(2)}_{z_0}(t,X)$ (see \eqref{metap}) satisfies
 $$ \dot{S}_t(z_0)= JF(t)S_t(z_0),\quad S_0=I_d.
 $$
 We now focus on the $\hbar$-dependent equation \eqref{metaph} and consider the Gaussian beam defined in \eqref{TE}. We start with the  ansatz
\begin{equation}\label{Riccati}
\sht(z_0)\phi^\hbar_0(x) = a(t) e^{ \frac{i}{2\hbar} \Gamma_t x\cdot x},
\end{equation}
where $\Gamma_t\in\Sigma_d$, the Siegel space of complex symmetric matrices $\Gamma$ such that $\Im \Gamma>0$ (for details we refer, e.g., to \cite[Chapter 5]{folland89}) and $a(t)$ is a complex valued time dependent function.
Observe that $\Gamma_t$ must satisfy a Riccati equation and $a(t)$ a linear differential equation. Indeed, imposing that the right-hand side $\psi_t(x)=a(t) e^{ \frac{i}{2\hbar} \Gamma_t x\cdot x}$ of \eqref{Riccati}  satisfies the equation
\begin{equation}\label{metaphf}
 i \hbar \partial_t \psi_t(x) =\widehat{H^{(2)}_{z_0}(t)} \psi_t(x),\quad \psi_0(x)=\phi_0^\hbar(x),
 \end{equation}
it follows that the matrices $\Gamma_t$ fulfill the following Riccati equation
\begin{equation}\label{Riccatigenerale}
\dot{\Gamma}_t = - K - \Gamma_t L_t - L^T_t\Gamma_t-\Gamma_tG\Gamma_t,
\end{equation}
 with the initial conditions $\Gamma_0= i I_d$,
 whereas the function $a(t)$ satisfies
 \begin{equation}\label{at}
 \dot{a}(t)=-\frac12 \Tr(L^T_t+G_t\Gamma_t)a(t)
 \end{equation}
  with initial condition $a(0)=(\pi \hbar)^{-d/4}$.
Let us introduce the  matrices  $M_t=A_t+iB_t$, $N_t=C_t+iD_t$ which are nonsingular (see \cite[Chapter 5]{folland89}) and fulfill by \eqref{He}
\begin{align}
\dot{M}_t&=L_t M_t+G_t N_t\label{Mt}\\
\dot{N}_t&=-K_t M_t-L^T_t N_t\label{Nt},
\end{align}
$M_0=I_d$, $N_0=i I_d$. Furthermore, it can be proved the equality
$$ \Gamma_t=N_t M_t^{-1}.$$
The solution of \eqref{at} can be computed easily. Indeed, using \eqref{Mt}, we observe that $$\Tr(L^T_t+G_t\Gamma_t)=\Tr(\dot{M}_t M_t^{-1})$$ and  the Liouville formula
 $$\frac {d}{dt} \log (\det M_t)=\Tr(\dot{M}_t M_t^{-1}),
 $$
yields at once the  solution
\begin{equation}\label{ampl}
a(t)=(\pi \hbar)^{-d/4} (\det M_t)^{-1/2}.
\end{equation}
Finally,  we rephrase \eqref{TE} as
\begin{equation}\label{TEbis} \phi^{\hbar,0}_{z_0}(t,x)=e^{\frac{i}{\hbar}\tilde{\delta}(t,z_0)} e^{\frac{i}{\hbar} p_t(x-x_t)}(\sht(z_0)\phi^\hbar_0)(x-x_t),\\
\end{equation}
where
$$\tilde{\delta}(t,z_0)=\frac12 x_0p_0+\int_0^t p_s \dot{x}_s-H(s,z_s)\,ds
$$
so that
\begin{equation}\label{deltat}\dot{\tilde{\delta}}(t,z_0)= p_t \dot{x}_t-H(t,z_t)=p_t \partial_pH(t,x_t,p_t)-H(t,x_t,p_t).
\end{equation}

\smallskip
We now come back to the example  \eqref{gen_Harm_osc} for which
$$ F(t)=\begin{pmatrix}\partial^2_x V(x_t)&0_d\\0_d&I_d\end{pmatrix}.
$$
The equations \eqref{HS}, \eqref{Mt}, \eqref{Nt}, \eqref{deltat} become
\begin{equation}\label{GB_eqt}
\begin{split}
\dot{x_t} &= p_t\hfill\\
\dot{p_t} &= -\partial_x V(x_t)\\
\dot{M}_t &{}= N_t
\\
\dot{N}_t &{}= -\partial^2_{x}V(x_t)M_t\\
\dot{\tilde{\delta}}(t,z_0) &{}= \frac{|p_t|^2}{2}-V(x_t).
\end{split}
\end{equation}
with the initial conditions
\[
\{x_0,p_0,  M_0 = I_d,  N_0 = iI_d,\tilde{\delta}(0,z_0) = \frac{p_0 x_0}{2}\}.
\]
These equations characterize the Gaussian beams \eqref{TEbis}.
\subsection{Construction of the parametrix}
We consider a $\hbar$-Gabor frame $\G^\hbar(\phi_0^\hbar,h^{1/2}\Lambda)$, with $\Lambda=\alpha \bZ^d\times \beta \bZ^d$, $\alpha,\beta>0$. Let $\gamma^h$ be a dual window in $\cS(\rd)$. Using Theorem \ref{mainteo}, an approximate solution with $N=0$ to the Cauchy problem
  \begin{equation}
\begin{cases} i \hbar
\partial_t u =  (V(x) -\frac{\hbar^2}{2}\Delta)u\\
u(0)=u_0,
\end{cases}
\end{equation}
is provided by the expansion \eqref{Gab_exp}:
\begin{equation}\label{Gab_exp0}
[U^{(0)}(t) u_0](t,x)=\sum_{z\in h^{1/2}\Lambda}\langle u_0,\th(z)\gamma^h\rangle \phi_z^{\hbar,0}(t,x).
\end{equation}
In the framework of nonlinear approximation, dealt in Theorem \ref{mainteo1},
we can fix a tolerance $\eta>0$, and consider the set:
\begin{equation}\label{thresh}
A_{\eta,u_0}=\{z\in h^{1/2}\Lambda: |\langle u_0,\th(z)\gamma^h\rangle|>\eta\}.
\end{equation}
In this case our parametrix becomes
\[
U^{(0)}_{\eta}(t)[u_0](t,x) =  \sum_{z\in A_{\eta,u_0}} c_z \phi_z^{\hbar,0}(t,x).
\]

\subsection{Algorithms}
We work in dimension $d=1$ and assume that the initial datum $u_0$ is a signal of length $L\in 2\bN$, defined on the periodized unit interval $I = [0,1]$.\\
The space and frequency grids are defined as
\[
X = \{0,1/L,\ldots (L-1)/L\}, \quad \Omega = \{-L/2,-(L-1)/2\ldots 0\ldots L/2-1\},
\]
(with periodic boundary conditions). For details and a complete exposition on the parameters involved in the algorithms we address to the LTFAT documentation in http://ltfat.sourceforge.net/.
\begin{Alg}[Coefficients, via Discrete Gabor Transfrom] \normalfont{Consider a signal $f$ of length $L$.

\begin{enumerate}
\item[1.] Define $a>0$ to be the length of the time shift. The space translation is $a/L$.
\item[2.] Define $M>0$ to be the number of channels. The frequency translation will be of length $L/M$. In order to have a frame, the density needs to be less than $1$, hence  the parameters $a$ and $M$ must be chosen such that $a M <1$.
\item[3.] Compute the dilated gaussian window $g=(\pi \hbar)^{-1/4}e^{-x^2/(2\hbar)}$ (recall $d=1$). There is a LTFAT routine available, which gives a periodic normalized Gaussian: $g =$ \textbf{pgauss}$(L,L h\pi)$.
\item[4.] Calculate the dual window. The LTFAT command is $g_d =$\textbf{gabdual}$(g,a,M,L)$.
\item[5.] Calculate the coefficients $c$ of the signal $f$ using the Discrete Gabor Transform $c=$\textbf{dgt}$(f,g_d,a,M)$.
By default it gives the following:
    \[
        c(m,n) = \sum_{l = 0}^{L-1} f(l) \overline{g_{d}(l-an)} e^{-2\pi i l m/M},
    \]
for $n = 0,\ldots, L/a-1, \: m = 0,\ldots,M-1$. Since, by \eqref{TEbis}, we need the expansion
    \[
        c(m,n) = \sum_{l = 0}^{L-1} f(l) \overline{g_{d}(l-an)} e^{-2\pi i(l-a n) m/M},
    \]
we apply the routine \textbf{phaselock}, i.e. $c = $\textbf{phaselock}$(c)$.
\end{enumerate}}
\label{Al:GabCoeff}
\end{Alg}
We can use this algorithm as initial step for our procedure.
\begin{Alg}[Gaussian Beam, standard setting]
\normalfont{Consider the initial condition $u_0$ of length $L$.
\begin{enumerate}
\item[1.] Use Algorithm \ref{Al:GabCoeff} to compute the coefficients $c$ of $u_0$.
\item[2.] Set a threshold $\eta>0$, using \textbf{thresh}$(c, \eta)$.
\item[3.] Set the initial values for the ODEs (cf.\eqref{GB_eqt}). The initial displacement is $x_0=a n/L$ where, for symmetry reasons, \begin{equation}\label{rangen} n =-L/2+1,\ldots,-1,0,1, \ldots,L/2.\end{equation} This is not restrictive, since the ODE is solved modulus one.
 The initial momentum  is $p_0=2\pi h m L/M$, with \begin{equation}\label{rangem}m = -M/2 +1,\ldots,-1,0,1 \ldots, M/2.
   \end{equation} Set $$\tilde{\delta}(0,z_0)=(a n \pi h m) /M,\, \,M_0 =1,\,\, N_0 =i.$$
 To obtain the normalization of \textbf{pgauss} we set $a(0) = g(L)$.
\item[4.] Solve the set of ODEs for each $(m,n)$, where $m$ ranges over the integers in \eqref{rangem} and $n$ over the integers in \eqref{rangen}.  This can be done in MATLAB using a standard solver. In our implementation we used \textbf{ode45}.
\item[5.] Construct the solution $U^{(0)}_{\eta}(t)[u_0](t,x)$ using the \textbf{exp} function of MATLAB.
\end{enumerate}}
\label{Alg:GB_ODe}
\end{Alg}
\subsubsection{Large-time behavior}
The unique feature of Gaussian Beams, or nearly coherent states, is that they are effective even in the presence of caustics.
Nevertheless, in certain cases, the large-time behavior can be troublesome. Precisely, the smallest eigenvalues of the matrix $\Gamma_t$ can drop quickly and the corresponding  Gaussian in \eqref{Riccati} starts to spread. This leads to a drop of quality in our  solution. This phenomena is studied in \cite{Lex2} in dimension $d=1$. The authors  relate the sign of the Hessian of the potential $V(x)$ to the spreading of the Gaussian in time. When the Hessian is positive, i.e. the so-called \textbf{potentil well}, the matrix $\Gamma_t$ is bounded. When it is constant, $\Gamma_t$ shows a linear decay in time and when the Hessian is negative,  $\Gamma_t$ decays exponentially. The latter case is named \textbf{potential hill} and it needs to be treated very carefully.
 Although this treatment is far from being conclusive for the general case, it suits our $1-$dimensional problem. Hence, we follow their idea of \emph{reinitialization}.\par
  We monitor the decay of $\Gamma_t$ and as soon as it drops under a certain tolerance, we stop the propagation at time $\overline{t}$, say. We then compute the solution $u_{\overline{t}}$ and  use it as initial value for the evolution in the time interval $[\,\overline{t},T]$. Let us describe the algorithm in detail.

\begin{Alg}[Gaussian Beam, reinitialization]
\label{Alg:GB_Re}

\normalfont{Consider the initial condition $u_0$ of lenght $L$.
\begin{enumerate}
\item[1.] Set $U_r = u_0$, where $u_0$ is the initial value as before.%
\item[2.] Compute the Gabor coefficients of $U_r$, using Algorithm \ref{Al:GabCoeff}.
\item[3.] Solve the ODEs, setting the initial values as in Algorithm \ref{Alg:GB_ODe}. Use \textbf{`Events'} inside \textbf{odeset} to monitor the decay of the fourth component. With this command, we can set a tolerance and, if the matrix $\Gamma_t$ drops below that, the computation of \textbf{ode45} stops and return the time $\overline{t}$ together with the solutions at that time.
\item[4.] Construct the solution $U^{(0)}_{\eta}(t)[u_0](t,x)$ at time $t = \overline{t}$.\\
\item[5.] Set  $U_r = U^{(0)}_{\eta}(\overline{t})[u_0](\overline{t},x)$ and execute again steps 1-4 until $\overline{t}$ reaches the final time.
\end{enumerate}}
\end{Alg}
\subsection{Numerical Results}
The problems we consider are the ones presented in \cite{Lex2} and, for sake of consistency,  the same comparison method is used, i.e. the  Strang Splitting pseudo-spectral method \cite{MR1880116,MR2047194}.
\subsubsection{Potential well}
We consider the potential
\begin{equation}\label{es1}
V(x) = \cos(2\pi x),
\end{equation}
treated in the Example $1$ in \cite[Subsec. 6.1.1]{Lex2} and set the initial values
\begin{equation}\label{In_dat}
\begin{split}
u_0(x) &{}= {{ e}^{-25\, \left( x- 0.5 \right) ^{2}}}{{ e}^{\frac{i\tau_0(x)}{h}}}\\
\tau_0(x)&{}=-\frac{1}{5}\log(e^{(5(x-0.5))}+e^{-5(x-0.5)});
\end{split}
\end{equation}
with signal length $L=1024$ and Plank constant $h=\frac{1}{256\pi}$.
We also set a threshold for the coefficients (cf. \eqref{thresh}) taking the ones with absolute value greater than $\eta = 0.01$. With this tolerance, our initial reconstruction is still accurate. Indeed, at time $t=0$ the relative error is just of order $0.3\%$, similar to the one  in the Example $1$ in \cite[Subsec. 6.1.1]{Lex2}.\par
 The potential \eqref{es1} is non negative for $x\in [0.25,0.75]$. So we expect the solution to be accurate inside this interval, since the beam width is constant. This is consistent with the results of Figure \ref{fig:well}, where no reinitialization is needed.\\

\begin{figure}
        \centering
        \begin{subfigure}[b]{0.4\textwidth}
        \includegraphics[width=\textwidth]{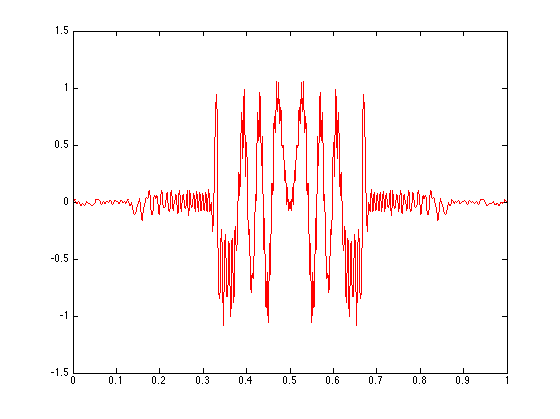}
        \caption{Strang-Splitting solution}
        \end{subfigure}
         \begin{subfigure}[b]{0.4\textwidth}
                \includegraphics[width=\textwidth]                {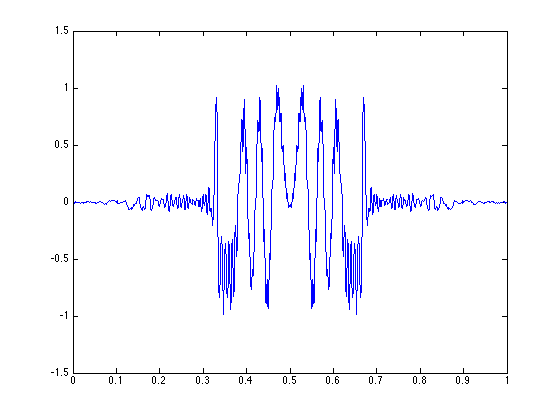}
                \caption{Beam solution}
                        \end{subfigure}
         \begin{subfigure}[b]{0.4\textwidth}
                \includegraphics[width=\textwidth]{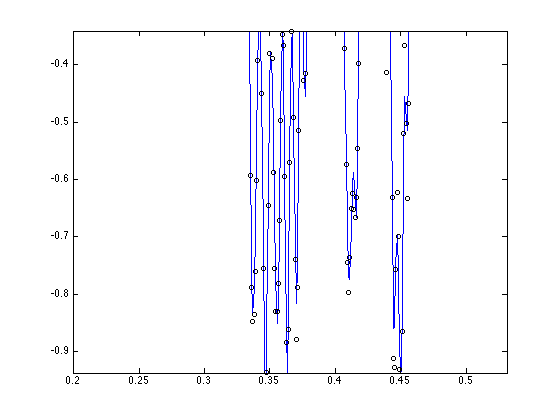}
                \caption{Comparison, center.}
                        \end{subfigure}
                        \begin{subfigure}[b]{0.4\textwidth}
                \includegraphics[width=\textwidth]{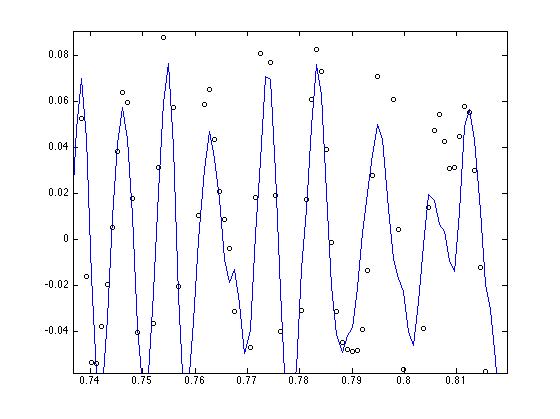}
        \caption{Comparison, side.}\end{subfigure}
        \caption{Potential $V(x) = \cos(2\pi x)$. (A) The exact solution. (B) The beam solution. (C) and (D) Comparison of the exact  and beam solution.   ``$\circ$": the real\ solution; ``$-$": the beam solution.}\label{fig:well}
\end{figure}
We can push our analysis a little bit further. In Figure \ref{fig:IV} our initial values are represented by the picture on the left-hand side, whereas on the right-hand side we show the same function multiplied by a narrower Gaussian. This provides an ``almost'' compactly supported function inside the interval $[0.25,0.75]$ where we have the potential well. In this case we obtain a $35\%$ drop on the relative error. The solution is shown in Figure \ref{fig:well_CSP}, which contains a comparison of the exact and beam solution.
\begin{figure}
        \centering
        \begin{subfigure}[b]{0.4\textwidth}
        \includegraphics[width=\textwidth]{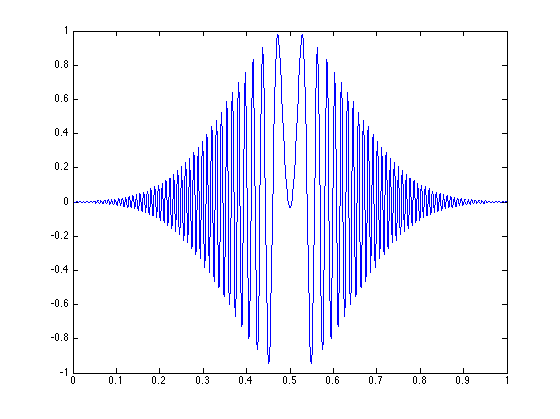}
        \end{subfigure}
         \begin{subfigure}[b]{0.4\textwidth}
        \includegraphics[width=\textwidth]{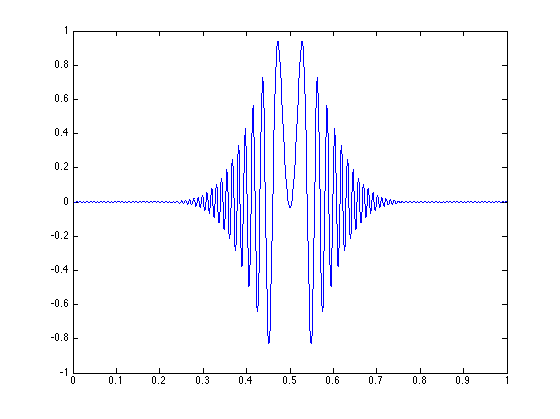}
        \end{subfigure}
\caption{Initial values. The right one is numerically supported in $[0.25, 0.75]$}
\label{fig:IV}
\end{figure}
\begin{figure}
        \centering
        \begin{subfigure}[b]{0.4\textwidth}
        \includegraphics[width=\textwidth]{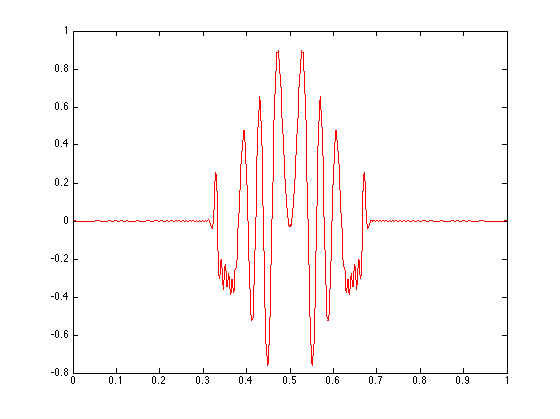}
        \caption{Strang-Splitting solution}
        \end{subfigure}
         \begin{subfigure}[b]{0.4\textwidth}
                \includegraphics[width=\textwidth]                {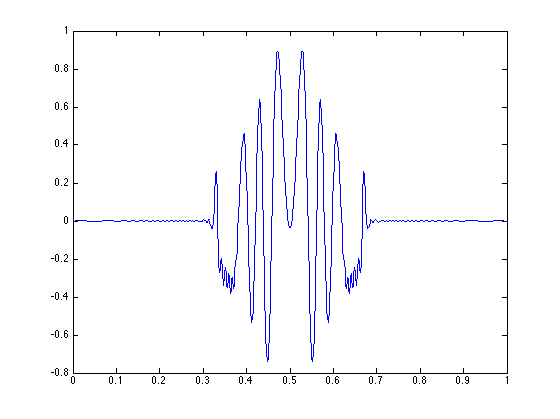}
                \caption{Beam solution}
                        \end{subfigure}
         \begin{subfigure}[b]{0.4\textwidth}
                \includegraphics[width=\textwidth]{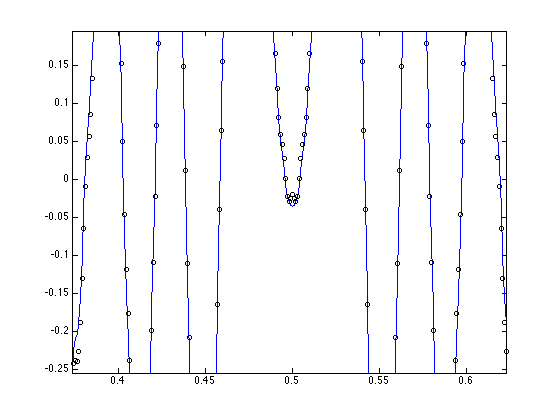}
                \caption{Comparison, center.}
                        \end{subfigure}
                        \begin{subfigure}[b]{0.4\textwidth}
                \includegraphics[width=\textwidth]{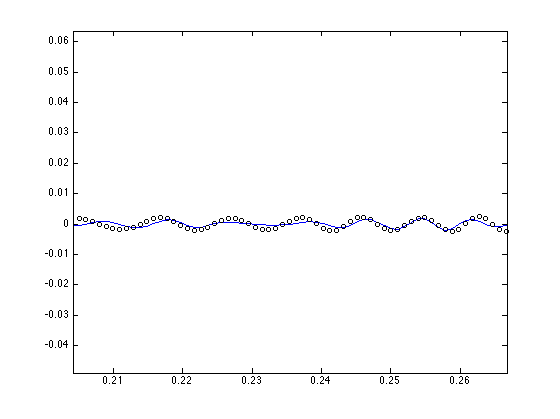}
        \caption{Comparison, side.}\end{subfigure}
        \caption{Potential $V(x) = \cos(2\pi x)$, compactly supported initial values. (A) The exact solution. (B) The beam solution. (C) and (D) Comparison of the exact  and beam solution.   ``$\circ$": the real\ solution; ``$-$": the beam solution.}\label{fig:well_CSP}
\end{figure}
\subsubsection{A potential Hill}
Consider the initial value to be \eqref{In_dat}, we take the signal length to be $L=1024$ and the Plank constant $h=\frac{1}{256\pi}$, as above.
We also set the same threshold  $\eta = 0.01$.
We consider the potential $V(x) = \cos(2\pi(x+0.5))$, whose Hessian is non-positive for $x \in [0.25,0.75]$. This yields unacceptable numerical results for the standard algorithm, as shown in Figure \ref{fig:hill}.
\begin{figure}
        \centering
        \begin{subfigure}[b]{0.4\textwidth}
        \includegraphics[width=\textwidth]{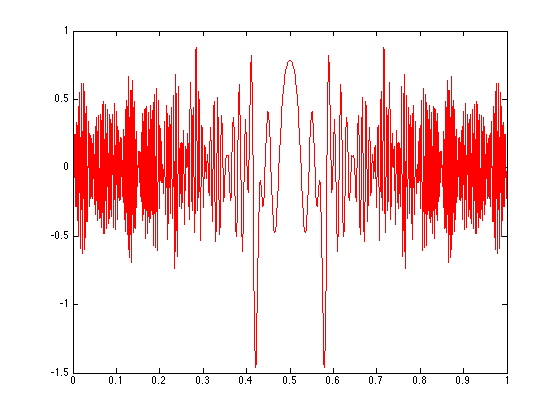}
        \caption{Strang-Splitting solution}
        \end{subfigure}
         \begin{subfigure}[b]{0.4\textwidth}
                \includegraphics[width=\textwidth]                {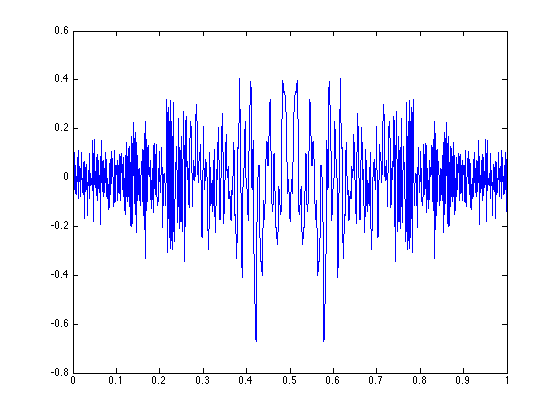}
                \caption{Beam solution}
                        \end{subfigure}
         \begin{subfigure}[b]{0.4\textwidth}
                \includegraphics[width=\textwidth]{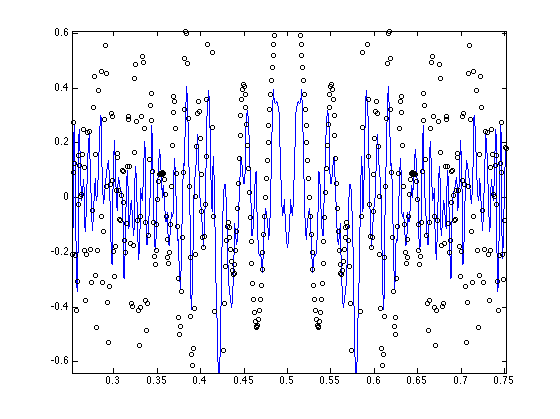}
                \caption{Comparison, center.}
                        \end{subfigure}
                        \begin{subfigure}[b]{0.4\textwidth}
                \includegraphics[width=\textwidth]{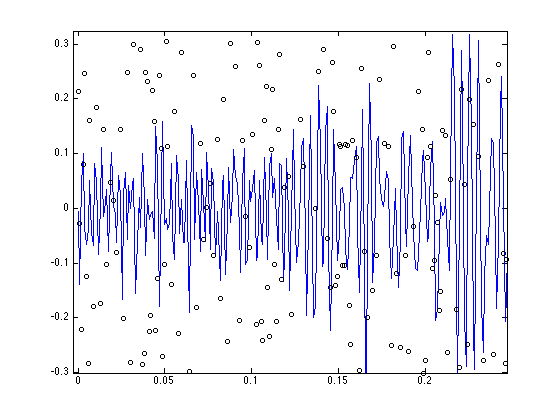}
        \caption{Comparison, side.}\end{subfigure}
        \caption{Potential $V(x) = \cos(2\pi (x+0.5))$. (A) The exact solution. (B) The beam solution. (C) and (D) Comparison of the exact  and beam solution.   ``$\circ$": the real\ solution; ``$-$": the beam solution.}\label{fig:hill}
\end{figure}
If we use the reinitialization algorithm, then our approximation improves a lot, see Figure \ref{fig:hillRE}. In this case we subdivide the time interval $[0,2]$ in eight uniform subintervals.
\begin{figure}[ht!]
        \centering
        \begin{subfigure}[b]{0.4\textwidth}
        \includegraphics[width=\textwidth]{Strang2.png}
        \caption{Strang-Splitting solution}
        \end{subfigure}
         \begin{subfigure}[b]{0.4\textwidth}
                \includegraphics[width=\textwidth]                {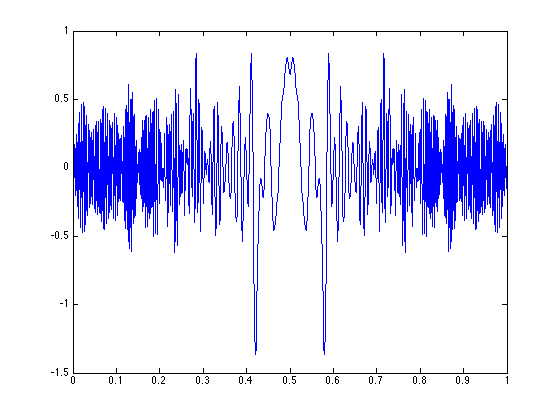}
                \caption{Beam solution}
                        \end{subfigure}
         \begin{subfigure}[b]{0.4\textwidth}
                \includegraphics[width=\textwidth]{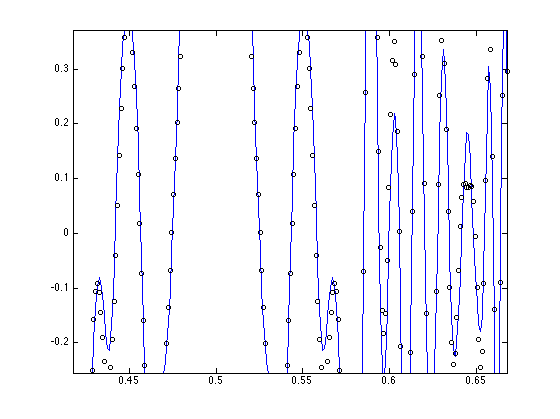}
                \caption{Comparison, center.}
                        \end{subfigure}
                        \begin{subfigure}[b]{0.4\textwidth}
                \includegraphics[width=\textwidth]{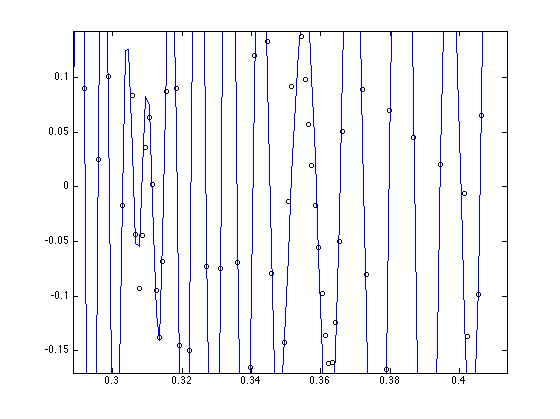}
        \caption{Comparison, side.}\end{subfigure}
        \caption{Potential $V(x) = \cos(2\pi (x+0.5))$ with reinitialization. (A) The exact solution. (B) The beam solution. (C) and (D) Comparison of the exact  and beam solution.   ``$\circ$": the real\ solution; ``$-$": the beam solution.}\label{fig:hillRE}
\end{figure}
 \begin{figure}
 \centering
 \includegraphics[width=0.6\textwidth]{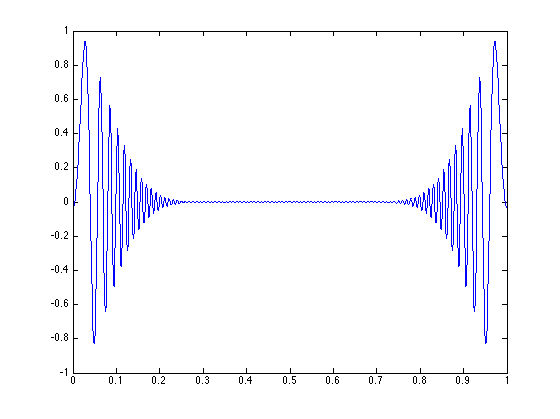}
        \caption{Shifted initial value}
\label{fig:ffts}
\end{figure}
We can try to do as before and chose an initial value which is compactly supported in $[0,0.25]\cup[0.75,1]$, i.e. where the Hessian is non negative. The easiest way to do it is to shift the initial value used above by half the length of the signal, see Figure \ref{fig:ffts}.
Then, as shown in Figure \ref{fig:T2}, even without any reinitialization we get a perfect reconstruction, as expected. We notice that this is nothing but the first case shifted by half the wavelength.
\begin{figure}[ht!]
        \centering
        \begin{subfigure}[b]{0.4\textwidth}
        \includegraphics[width=\textwidth]{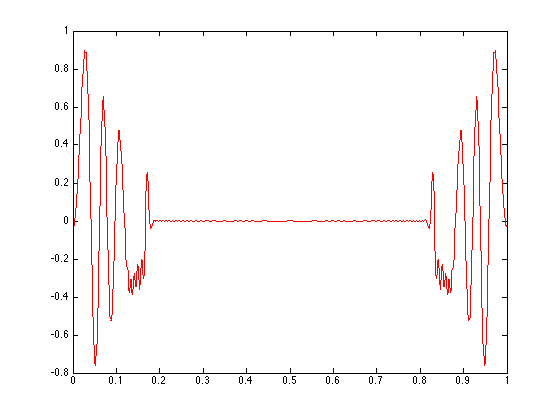}
        \caption{Strang-Splitting solution}
        \end{subfigure}
         \begin{subfigure}[b]{0.4\textwidth}
                \includegraphics[width=\textwidth]                {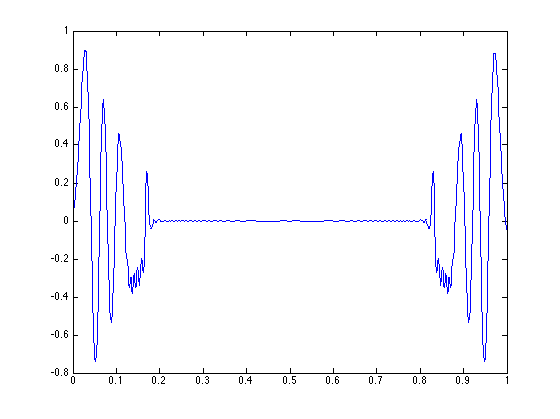}
                \caption{Beam solution}
                        \end{subfigure}
\caption{Potential $V(x) = \cos(2\pi (x+0.5))$, compactly supported initial values.}
\label{fig:T2}
\end{figure}
\subsubsection{A potential hill and well}
Consider the initial value to be \eqref{In_dat}, we take the signal length to be $L=1024$ and the Plank constant $h=\frac{1}{256\pi}$, as above.
We also set the same threshold  $\eta = 0.01$.\\
We take $V(x) = 10 +\sin(2\pi(x+0.5))$, whose Hessian is non-positive for $x \in [0,0.5]$. Once again, the numerical results without reinitialization are far from being consistent, see Figure \ref{fig:hillwell}.
\begin{figure}
        \centering
        \begin{subfigure}[b]{0.4\textwidth}
        \includegraphics[width=\textwidth]{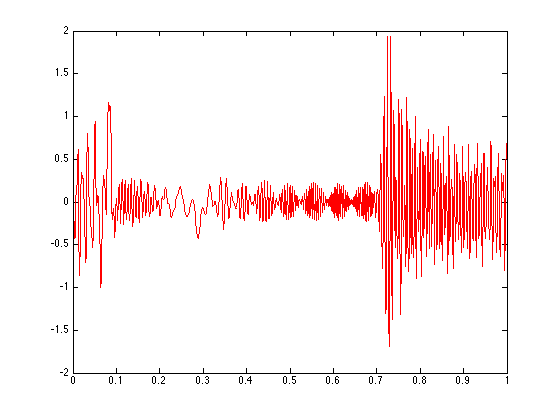}
        \caption{Strang-Splitting solution}
        \end{subfigure}
         \begin{subfigure}[b]{0.4\textwidth}
                \includegraphics[width=\textwidth]                {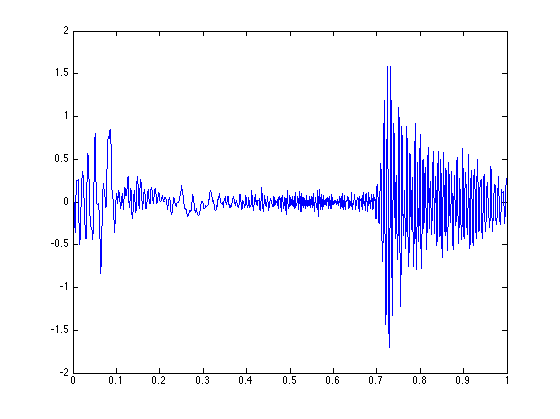}
                \caption{Beam solution}
                        \end{subfigure}
         \begin{subfigure}[b]{0.4\textwidth}
                \includegraphics[width=\textwidth]{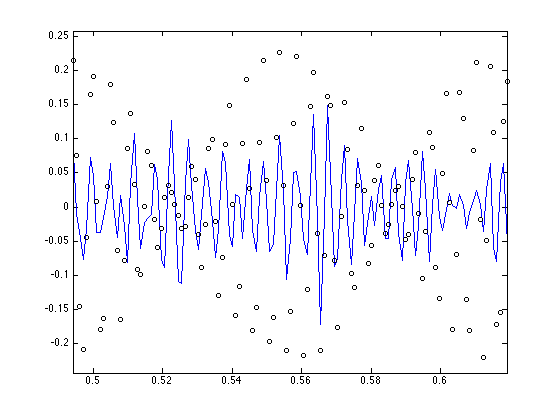}
                \caption{Comparison, center.}
                        \end{subfigure}
                        \begin{subfigure}[b]{0.4\textwidth}
                \includegraphics[width=\textwidth]{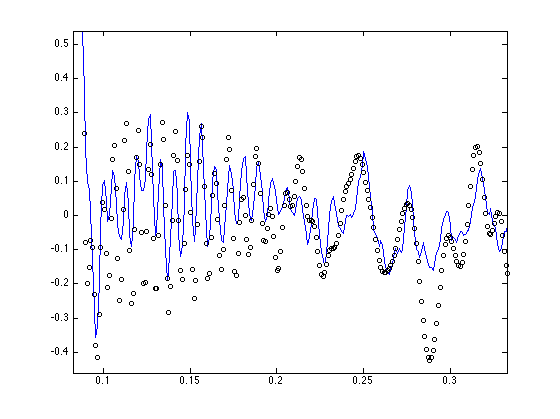}
        \caption{Comparison, side.}\end{subfigure}
        \caption{Potential $V(x) = 10 +\sin(2\pi(x+0.5))$. (A) The exact solution. (B) The beam solution. (C) and (D) Comparison of the exact  and beam solution.   ``$\circ$": the real\ solution; ``$-$": the beam solution.}\label{fig:hillwell}
\end{figure}
If we use the reinitialization, our results improve greatly as shown in Figure \ref{fig:hillwell_re}.\medskip\\
\begin{figure}[ht!]
        \centering
        \begin{subfigure}[b]{0.4\textwidth}
        \includegraphics[width=\textwidth]{Strang3.png}
        \caption{Strang-Splitting solution}
        \end{subfigure}
         \begin{subfigure}[b]{0.4\textwidth}
                \includegraphics[width=\textwidth]                {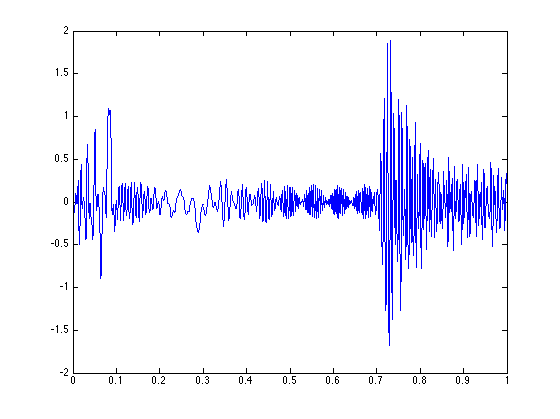}
                \caption{Beam solution}
                        \end{subfigure}
         \begin{subfigure}[b]{0.4\textwidth}
                \includegraphics[width=\textwidth]{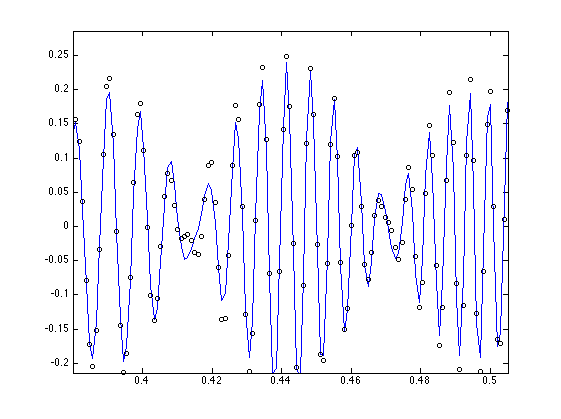}
                \caption{Comparison, center.}
                        \end{subfigure}
                        \begin{subfigure}[b]{0.4\textwidth}
                \includegraphics[width=\textwidth]{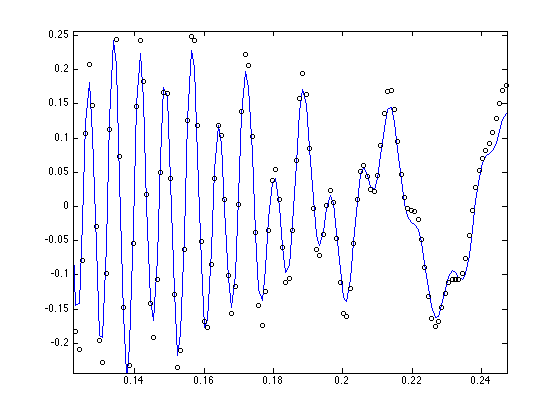}
        \caption{Comparison, side.}\end{subfigure}
        \caption{Potential $V(x) = 10+\sin(2\pi (x+0.5))$ with reinitialization. (A) The exact solution. (B) The beam solution. (C) and (D) Comparison of the exact  and beam solution.   ``$\circ$": the real\ solution; ``$-$": the beam solution.}\label{fig:hillwell_re}
\end{figure}
We can try a similar trick as before and pick the usual compactly supported window shifted in the interval $[0.5,1]$, once again the results are very good, see Figure \ref{fig:T3}.

\begin{figure}[ht!]
        \centering
        \begin{subfigure}[b]{0.4\textwidth}
        \includegraphics[width=\textwidth]{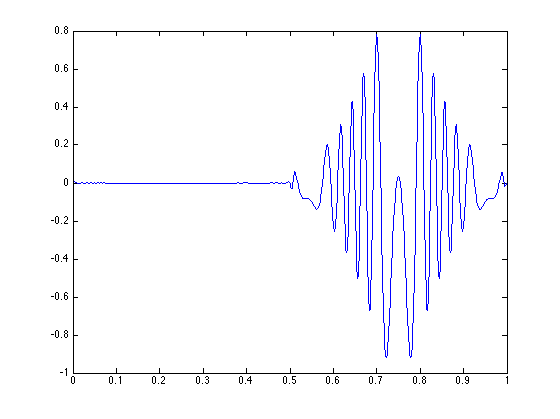}
        \caption{Strang-Splitting solution}
        \end{subfigure}
         \begin{subfigure}[b]{0.4\textwidth}
                \includegraphics[width=\textwidth]                {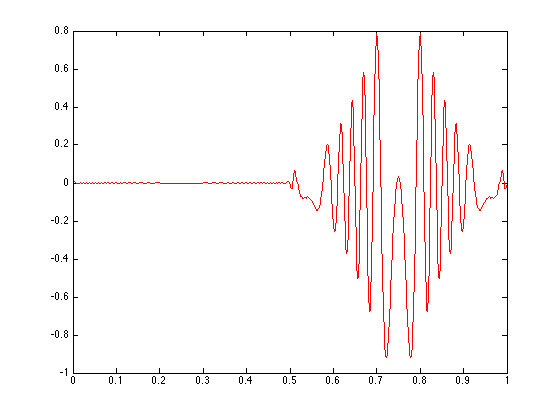}
                \caption{Beam solution}
                        \end{subfigure}
\caption{Potential $V(x) = 10 + \sin(2\pi (x+0.5))$, compactly supported initial values.}
\label{fig:T3}
\end{figure}
\subsection{Future Works}
The LTFAT package provides a 2-dimensional discrete Gabor transform (\textbf{dgt2}), although the \textbf{phaselock} command is not yet available. Nevertheless, it seems possible to implement a multidimensional algorithm using the same approach. We plan to develop such a command and  then investigate the two and  three-dimensional cases.\\
\section*{Acknowledgments}
We would like to thank Lexing Ying and Janling Qian for the useful discussions on the algorithms. We also thank Nicki Holigaus for his help in the use of the LTFAT package. \par The first three authors were partially supported by  the Gruppo
Nazionale per l'Analisi Matematica, la Probabilit\`a e le loro
Applicazioni (GNAMPA) of the Istituto Nazionale di Alta Matematica
(INdAM).

\addcontentsline{toc}{chapter}{References}
\bibliography{Bib_GB_Schr}
\bibliographystyle{abbrv}
\end{document}